\documentclass[12pt]{article}
\usepackage{amsthm,amsmath,amssymb,latexsym}
\usepackage[textsize=footnotesize]{todonotes}
\usepackage{xcolor}
\usepackage{xspace}
\usepackage{authblk}
\usepackage{indentfirst}

\newtheoremstyle{alg}
    {3pt}
    {3pt}
    {}
    {}
    {\bfseries}
    {.}
    {\newline}
    {}

\usepackage{hyperref}
\hypersetup{
    pdftitle={Fast recognition of alternating groups of unknown degree},
    pdfauthor={Sebastian Jambor, Martin Leuner, Alice C. Niemeyer, and Wilhelm Plesken},
    colorlinks=false,
    pdfborder={0 0 0}
}

\newtheorem{theorem}{Theorem}[section]
\newtheorem{lemma}[theorem]{Lemma}
\newtheorem{proposition}[theorem]{Proposition}
\newtheorem{corollary}[theorem]{Corollary}

\theoremstyle{definition}
\newtheorem{remark}[theorem]{Remark}

\theoremstyle{alg}
\newtheorem{algorithm}[theorem]{Algorithm}

\newcommand{\N}{\mathbb{N}}
\newcommand{\veps}{\varepsilon}
\newcommand{\wt}[1]{\widetilde{#1}}

\newcommand{\diff}[1]{\ensuremath{\frac{\operatorname{d}}{\operatorname{d}\!{#1}}}}
\newcommand{\dint}[1]{\ensuremath{\operatorname{d}\!{#1}}}
\newcommand{\vnot}[2]{\ensuremath{\operatorname{#1}_{\neg#2}}}
\newcommand{\mcO}{\mathcal{O}}

\newcommand{\algTC}{\textsc{ThreeCycleCandidates}\xspace}
\newcommand{\algBE}{\textsc{BolsteringElements}\xspace}
\newcommand{\algBC}{\textsc{BuildCycle}\xspace}
\newcommand{\algLC}{\textsc{ConstructLongCycle}\xspace}
\newcommand{\algDFP}{\textsc{IsFixedPoint}\xspace}
\newcommand{\algFPM}{\textsc{AdjustCycle}\xspace}
\newcommand{\algEC}{\textsc{AppendPoints}\xspace}

\newcommand{\algGens}{\textsc{StandardGenerators}\xspace}
\newcommand{\algMain}{\textsc{RecogniseSnAn}\xspace}

\DeclareMathOperator{\Aut}{Aut}
\DeclareMathOperator{\Sym}{S}
\DeclareMathOperator{\Alt}{A}
\DeclareMathOperator{\supp}{supp}
\DeclareMathOperator{\prob}{P}

\DeclareMathOperator{\fix}{fix}
\DeclareMathOperator{\Centra}{C}
\DeclareMathOperator{\inv}{inv}
\DeclareMathOperator{\trip}{trip}

\allowdisplaybreaks[4]

\colorlet{details}{black!30} 

\newcommand{\CompileWithDetails}{0}			
\newcommand{\details}[1]{\ifthenelse{\equal{\CompileWithDetails}{1}}{%
	\begin{quote}
    \color{details}
	\textbf{Details:}
	#1
	\end{quote}
}{}}
\newcommand{\detailsalign}[1]{ \ifthenelse{\equal{\CompileWithDetails}{1}}{%
& \color{details} #1  \\
}{} } 
\newcommand{\detailsalt}[2]{ \ifthenelse{\equal{\CompileWithDetails}{1}}{#1}{#2}}%
\newcommand{\detailsinline}[1]{ \ifthenelse{\equal{\CompileWithDetails}{1}}{\color{details} #1 \color{black}}{}}%
\newcommand{\onlyshortversion}[1]{ \ifthenelse{\equal{\CompileWithDetails}{1}}{}{#1} }

\title{Fast recognition of alternating groups of unknown degree}
\author[1]{Sebastian Jambor\thanks{jambor@math.auckland.ac.nz}}
\author[2]{Martin Leuner\thanks{leuner@momo.math.rwth-aachen.de}}
\author[3]{Alice C.~Niemeyer\thanks{alice.niemeyer@uwa.edu.au}}
\author[2]{Wilhelm Plesken\thanks{plesken@momo.math.rwth-aachen.de}}
\affil[1]{
Department of Mathematics,
The University of Auckland
}
\affil[2]{
Lehrstuhl B f\"ur Mathematik,
RWTH Aachen University
}
\affil[3]{
Lehrstuhl D f\"ur Mathematik,
RWTH Aachen University\\
and
Centre for the Mathematics of Symmetry and Computation\\
University of Western Australia
}

\begin{document}

\date{}
\maketitle			

\begin{abstract}
We present a constructive recognition algorithm to decide
whether a given
black-box group is isomorphic to an alternating or a symmetric group without
prior knowledge of the degree.
This eliminates the major gap in known algorithms, as they require the
degree as additional input.

Our methods are probabilistic and rely on results about proportions
of elements with certain properties in alternating and symmetric groups.
These results are
of independent interest;
for instance, we establish a lower bound for the proportion of involutions with small support.
\end{abstract}

%
%

\section{Introduction}

The computational recognition of finite simple groups
is a fundamental task in the finite matrix group recognition project
(see \cite{LG,NS,OBrien}).
Generally not much is known about the
way in which a group might be given as input
and therefore algorithms which take
black-box groups (see \cite{Babai}) as input
are the  most versatile. For the important infinite family of
alternating groups,
the present black-box algorithms \cite{Bealsetal,BratPak}
can only test whether a given black-box group is isomorphic
to an alternating or a symmetric group of a particular degree,
provided as additional input to the algorithm. Therefore deciding
whether a given black-box group is isomorphic to an alternating group
may require to run the algorithm once for each possible degree.
The present paper describes a one-sided Monte-Carlo
(see e.g. \cite[p.~14]{akos})
black-box algorithm which avoids this bottleneck.
Our algorithm takes as input a black-box group
given by a set of generators together with a natural number $N$ and
decides whether the given group is isomorphic to an alternating group
of any degree at most $N$. If the algorithm proves this to be the case,
it computes the degree of the group and recognises it constructively.
Otherwise the algorithm reports failure.
Our algorithm runs in time nearly linear in $N$
whereas the older algorithms have a runtime complexity of $\widetilde{\mcO}(N^2)$
to solve the same task in the worst case.

Given a black-box group $G$,
we let $\mu$ denote an upper bound for the cost of multiplying two
elements in $G$ and let $\rho$ denote an upper bound for the cost of
computing a uniformly distributed, independent random element of~$G$.
Throughout this paper, $\log$ denotes the natural logarithm.

\begin{theorem}\label{thm:main}
Algorithm~\ref{alg:Main}, \algMain,
is a one-sided Monte-Carlo algorithm
with the following properties.
It takes as input a black-box group $G = \langle X \rangle$,
a natural number $N$ and a real number $\veps$ with
$0 < \varepsilon < 1$.
If $G\cong \Alt_n$ or $G \cong \Sym_n$ for some $9 \le n \le N$,
it returns with probability at least $1-\varepsilon$ the
degree $n$ and an isomorphism $\lambda: G \rightarrow \Alt_n$ or
$\lambda: G \rightarrow \Sym_n$. Otherwise it reports failure.
The algorithm runs in time $\mcO( N \log(N)^2 \log(\veps^{-1}) (|X|\mu + \rho))$ and
stores at most $\mcO( \log (N) )$ group elements at any moment.
\end{theorem}
%
%
%

The black-box construction of a $3$-cycle -- one of the key ingredients of the algorithm -- is a surprisingly hard problem.
The solution lies in the combination of the following theoretical results,
which are also of independent interest.
The first allows us to find involutions with small support;
the second uses these to construct a $3$-cycle.

\begin{theorem} \label{thm:SmallSupport}
Let $9 \leq n \in \mathbb{N}$ and $G \in \{\Alt_n,\Sym_n\}$.
The proportion of elements $x \in G$ of even order
satisfying
$\big| \supp x^{|x|/2} \big| \le 4\sqrt{n}/3 $
is at least
$\left( 13 \log(n) \right) ^ {-1}$.
\end{theorem}

\begin{theorem} \label{thm:InvToCyc}
Let $7 \le n \in \mathbb{N}$, $G \in \{\Alt_n,\Sym_n\}$ and $1 \le k \le 2\sqrt{n}/3$.
Let $s \in G$ be an involution moving $2k$ points.
\begin{enumerate}
\item The proportion of elements $r$ in the conjugacy class $s^G$
such that $r$ and $s$ move exactly one common point is at least $10/(3n)$.
\item Let $M$ be the set of elements in $s^G$ not commuting with $s$.
The proportion of
elements $r$ in $M$ such that $(sr)^2$ is a $3$-cycle
is at least $1/3$.
\end{enumerate}
\end{theorem}

The constructive recognition algorithm for alternating and symmetric
groups described in \cite{Bealsetal} consists of two parts: the
construction of standard generators assuming the degree is known, and
the algorithmic construction of the inverse of the isomorphism
$\lambda: G \rightarrow \Alt_n$. The contribution of this paper is to
replace the first part by an algorithm determining the degree and
finding the standard generators simultaneously. Together with the
second part of \cite{Bealsetal}, this establishes the algorithm for
the main theorem above.
If one is interested in recognising the symmetric group rather than
the alternating group, the remarks of \cite{Bealsetal} apply and the
same complexity is achieved.

Our algorithm has been implemented in the computer algebra system \textsf{GAP} \cite{GAP}.
Comparisons of our implementation with the \textsf{GAP} implementation of the first part
of \cite{Bealsetal} show that our algorithm is a significant improvement.
Given as input a black-box group isomorphic to a symmetric or alternating
group, the new algorithm  establishes this fact and
determines the degree of the group
in about the same time that the old algorithm requires to decide
whether the input group is isomorphic to an alternating or symmetric
group of the specific degree given as part of the input.
In general, the old algorithm has to be run several times
to find the degree of the input group. Therefore, the new algorithm wins out
by a factor determined by the number of putative degrees
the old algorithm has to test.
The scope of our implementation depends on many factors, in particular
the way the group is represented. To give a very rough indication,
in the natural permutation representation
the present implementation can deal with degrees of around 10000.

In applications in the matrix group recognition project it is imperative that
the algorithm report failure quickly when the input group is not isomorphic
to an alternating nor a symmetric group. We tested the performance of our
algorithm when handed some examples of almost simple groups which are not
alternating or symmetric. In all these examples our algorithm reported failure
extremely fast.
This is mainly due to finding an element of order not existing in the symmetric
group of degree $N$, thus even proving that the group cannot be of the specified isomorphism
types (cf. remark after Algorithm~\ref{alg:tc}).

The practical performance of our algorithm exceeds its predicted performance as
the constants in our estimates of proportions of elements are too conservative,
notably in the proportion proved in Theorem~\ref{thm:SmallSupport}.
Further improvements of the performance could be achieved in situations where
an order oracle is available by lowering the a priori upper bound $N$.

As E.~O'Brien pointed out, our algorithm can also be applied to decide
whether the input group $G$ is a central extension of some (not
necessarily finite) abelian group by $\Alt_n$ or $\Sym_n$ by working with
$G/Z(G)$ as black-box group.

Here is a short overview of this paper.
We fix some notation in Section~\ref{sec:prelims} and give an outline of the algorithm in Section~\ref{sec:outline}.
In Section~\ref{sec:algo} we describe the setup in detail and
prove Theorem~\ref{thm:main}.
Finally, in Section~\ref{sec:techno} we give proofs of Theorems~\ref{thm:SmallSupport} and~\ref{thm:InvToCyc},
along with proofs of some technical results which are used in the proof of Theorem~\ref{thm:main}.

\section{Preliminaries} \label{sec:prelims}

This paper describes a constructive recognition algorithm
which decides whether a given
black-box group is isomorphic to an alternating or a symmetric group.
The notion of when a black-box group is
\emph{constructively recognisable}  is
defined in~\cite[Definition 1.1]{Bealsetal}. In particular, we note that if our
algorithm concludes that a given black-box group $G$ is indeed
isomorphic to an alternating group $\Alt_n$ or a symmetric group $\Sym_n$
of some degree $n$, then it also determines an isomorphism
$\lambda: G\rightarrow \Alt_n$ or
$\lambda: G\rightarrow \Sym_n$ and a pair $\{s,t\}$ of  generators for
$G$, called the \emph{standard generators of $G$}.
We call $\lambda$ together with the standard generators $\{s,t\}$
a \emph{constructive isomorphism}.

The standard generators for $\Alt_n$ chosen by the algorithm satisfy
the following presentations given by Carmichael~\cite{Carm}:
\begin{equation} \label{eq:PresAltEven}
	\left\{ s,t ~\middle|~ s^{n-2}=t^3=(st)^{n-1} = (t^{(-1)^k}s^{-k}ts^k)^2 = 1
		~~\text{for}~~
	1 \leq k \leq \frac{n-2}{2} \right\}
\end{equation}
for even $n > 3$ and
\begin{equation} \label{eq:PresAltOdd}
	\left\{  s,t ~\middle|~ s^{n-2}=t^3=(st)^{n} = (ts^{-k}ts^k)^2 = 1
		~~\text{for}~~
	1 \leq k \leq \frac{n-3}{2} \right\}
\end{equation}
for odd $n > 3$.

Examples of standard generators for $\Alt_n$ are $s
= (1,2)(3,4,\ldots, n)$ and $t = (1,2,3)$ for $n$ even, and $s = (3,4,\ldots,n)$ and $t = (1,2,3)$ for $n$ odd.

Our algorithm exploits information gained by considering the
cycle types of  permutations in symmetric groups.
Recall that  the \emph{cycle   type} of an element $g\in \Sym_n$ is
defined as $1^{a_1}\cdots n^{a_n}$ if $g$ contains
$a_i$ cycles of length $i$ for $1 \le i \le n.$
Note that for $n \ge 7$ we have $\Aut( \Alt_n ) = \Sym_n$,
so the cycle type is preserved by all automorphisms of $\Alt_n$.
Thus,  if $G$ is isomorphic to $\Alt_n$ or $\Sym_n$,
the  cycle type  of $\lambda(g)$ is independent of the choice of
isomorphism $\lambda$ from $G$ to $\Alt_n$ or $\Sym_n$.
This allows us to generalise the notion of cycle type to elements of
a black-box group $G$ isomorphic to $\Alt_n$ or $\Sym_n$.

During the course
of the algorithm, we may encounter subgroups $\Alt_k$
of~$\Alt_n$. For $k\ge 7$ and $k$ odd, given a $3$-cycle
$c \in \Alt_n$   we say that
a $k$-cycle  $g$ \emph{matches} $c$ if $\{gc^2, c\}$ are standard
generators for $\Alt_k$.
Note that in this case $g$ must be of the form $(u,v,w,\ldots)$, where
$c = (u,v,w)$ for $u, v, w \in \{1,\ldots, n\}$.

Let $\pi \in \Sym_n$.
Call a point $i$ with  $1 \le i \le n$ a \emph{moved point} of $\pi$
if $i^\pi \neq i$.
Call the set of moved points of $\pi$ the \emph{support} of $\pi$,
denoted $\supp \pi$.
Similarly, denote by $\fix \pi$ the set of fixed points of $\pi$, that
is $\{1,\ldots,n\} - \supp \pi$.

\section{Brief outline of the algorithm}\label{sec:outline}

We describe a one-sided Monte-Carlo algorithm which takes as input a
black-box group $G$, a real number $\varepsilon$ with
$0<\varepsilon < 1$ and a positive integer~$N$.
 The aim of the algorithm is to determine whether
there is an integer $n$ with $9\le n\le N$ such that $G$ is isomorphic
to $\Alt_n$ or
$\Sym_n$. In the following we describe the main steps of our
algorithm. We present this description under the assumption that the
algorithm is given a black-box group $G$ which is indeed isomorphic
via the unknown isomorphism $\lambda$ to
$\Alt_n$ or $\Sym_n$ for some $n\le N$ and describe the types of elements
in $G$ we seek to establish this fact. If the algorithm is handed
a black-box group not isomorphic to an alternating or symmetric
group, then one of the subsequent steps will fail to find the required
elements and the algorithm reports failure.

The algorithm consists of three main steps.
In the first step we compute a subset $R \subseteq G$ which contains a
$3$-cycle with  high probability. The details are presented in
Algorithm \algTC in
Section~\ref{sec:tc}. If no such
set $R$ was found, then we conclude that $G$ is not isomorphic to
$\Alt_n$ or $\Sym_n$ for any $n$ with $9 \le n\le N$ and terminate.

The second step  repeats the following basic step for each element
$c\in R$.
We may assume without loss of generality
that $\lambda(c) = (1,2,3)$ and we seek a $k$-cycle $g$ matching $c$
such that $k \ge 3n/4$.
The construction of $g$ is described in Algorithm \algLC in
Section~\ref{sec:lc}.
If no such
element $g$ was found, then we discard $c$ as a putative $3$-cycle and continue
with the next candidate for $c$ in $R$.
Otherwise, without loss of
generality,  we may assume that $\lambda(g) = (1,2,\ldots,k)$.

The third step, described in
Algorithm \algGens\ in
Section~\ref{sec:mainalg},
determines the degree $n$.
This step repeats a basic step which
computes random conjugates $r = g^x$ of $g$ for $x\in G$.
Note that by now we have derived some partial information about
$\lambda$, namely  $\lambda(c) = (1,2,3)$  and $\lambda(g) = (1,2,\ldots,k)$.
This allows us to decide whether
$\supp \lambda(g^x)$ contains hereto unseen points in which case
the basic step replaces $g$ by an element $g'$ such that
$\lambda(g') = (1,2,\ldots,\ell)$ for some $\ell > k.$
The third step repeats this basic step until it obtains an
$n$- or an $(n-1)$-cycle and constructs the standard generators for
$G$ from these.

Finally, we use methods from~\cite{Bealsetal} to
check whether we have found standard generators and compute
a constructive isomorphism.

\section{Details of the algorithm}\label{sec:algo}

In this section, the steps of the algorithm are described in detail.
Each step in turn is broken down into
one or more procedures. Each procedure is
designed to accept an arbitrary black-box group as input, which forces
the output to be fairly generic.
Therefore each procedure has an accompanying lemma which gives an interpretation
of the output if the input is in fact a symmetric or alternating group.
A second lemma determines the complexity, which is valid for arbitrary black-box
groups as input.

\subsection{Construction of possible $3$-cycles} \label{sec:tc}

The following algorithm constructs a set of putative $3$-cycles.
It is based on the simple observation that the product of two involutions
$t_1, t_2$ 
with $|\supp(t_1) \cap \supp(t_2)| = 1$ squares to a $3$-cycle.

\begin{algorithm}[\algTC] \label{alg:tc}
\textit{Input:} A group $G$, a real number
$0 < \veps < 1$ and $N \in \mathbb{N}$.\\
\textit{Output:} A set $R \subset G$ or \textsc{fail}.\\
\textit{Algorithm:}
\begin{enumerate}
\item \label{step:TCConsts}
Let $ M := \prod_{p} p^{\lfloor \log_p(N) \rfloor}$,
where the product is over all odd primes $p$ with $p \le N$.
Let $B := \lceil 13 \log(N) \log\bigl( 3/\veps \bigr) \rceil$,
$T := \lceil 3 \log\bigl( 3/\veps \bigr) \rceil$
and $C := \bigl\lceil 3NT/5 \bigr\rceil$.

\item \label{step:TCRndElts}
Choose $B$ random elements $r_1,\ldots,r_B \in G$
and set $t_i := r_i^M$ for $1 \leq i \leq B$.

\item \label{step:TCReduceTwoPart}
For each $t_i$, if there is a smallest $a \in \N$ such that $t_i^{(2^a)} = 1_G$ and
$a - 1 \le \log_2(N)$, then replace $t_i$ by $t_i^{(2^{a-1})}$.
Otherwise return \textsc{fail}.

\item \label{step:TCNonComm}
For each $t_i$ set $\Gamma_i := \emptyset$.
Repeat the following step at most $C$ times:
Choose a random conjugate $c$ of $t_i$.
If $t_ic \neq ct_i$ and $| \Gamma_i | < T$, then add $c$ to $\Gamma_i$.

\item \label{step:TCReturn}
Return $\bigcup_{i = 1}^B \{(t_ic)^2\,:\, c \in \Gamma_i\}$.
\end{enumerate}
\end{algorithm}

Note that if the algorithm returns \textsc{fail}, then  Step~3 has found
an element $g\in G$ such that $|g|$ cannot be the order of any element in
any group $\Sym_n$ for $n \le N$. Hence
$G$ is proven not to be isomorphic to
$\Alt_n$ or $\Sym_n$ for any $n \le N$.

\begin{lemma} \label{la:TCOutput}
Let $9 \le N \in \mathbb{N}$, $0 < \veps < 1$ and $G \in \{ \Sym_n, \Alt_n \}$ for some $9 \le n \le N$.
A call to Algorithm $\algTC(G,\veps, N)$ returns a subset $R$ of~$G$
and, with probability at least $1-\veps$,
$R$ contains a $3$-cycle in~$G$.
Moreover,
$|R| \le \lceil 13 \log(N) \log\bigl( 3/\veps \bigr) \rceil
\cdot \lceil 3 \log\bigl( 3/\veps \bigr) \rceil$.
\end{lemma}
\begin{proof}
Note that $M$ is an odd integer and that
for every $g \in G$ the element $g^M$ has even order or is trivial.
Therefore,
by Corollary~\ref{cor:SmallSuppElts}, with probability
at least $1-{\veps}/{3}$ one of the $t_i$ constructed in Step~\ref{step:TCRndElts}
has even order such that $t := t_i^{|t_i|/2}$
is a product of $k$ disjoint transpositions with
$k \le \lfloor \max\{2\sqrt{n}/3,2\}\rfloor$.
Let $X$ be a list of $C$ random conjugates of $t$. Then, with
probability at least $1-\veps/3$,
$X$ contains at least $T$ elements
which do not commute with $t$ by Corollary~\ref{cor:EnoughNonComm}.
Now let $\Gamma$ be a list of $T$ random conjugates of $t$ not commuting
with $t$. By Corollary~\ref{cor:TCgivenNC}
there is, with probability at least $1-\veps/3$, an element $c \in \Gamma$
such that $(tc)^2$ is a $3$-cycle.
Thus, with probability at least $(1-\veps/3)^3 \ge 1 - \veps$, the set
$R$ contains a $3$-cycle.
Since after Step~\ref{step:TCNonComm} we have $|\Gamma_i| \le T$, clearly
$|R| \le T\cdot B$ holds. This implies the claimed bound for $|R|$.
\end{proof}

\begin{lemma} \label{la:TCRuntime}
Let $G$ be a finite group, $0 < \veps < 1$ and $N \in \mathbb{N}$.
Then \algTC with input $G,\veps,N$ runs in
$\mcO ( N \log(N)^2 \log(\veps^{-1})^2 ( \mu + \rho ) )$
time and requires storage of
$\mcO ( \log( N ) \log( \veps^{-1} )^2 )$
group elements.
\end{lemma}
\details{actual complexity is
$\mcO( N \log(N) \log( \veps^{-1} ) ( \log(N) + \log(\veps^{-1}) ) \mu + N \log(N) \log( \veps^{-1} )^2 \rho )$}
\begin{proof}
Since $M < \prod_{2 < p \le N} N \le N^N$, computing the $M$-th power
of a group element  with a square-and-multiply algorithm requires $\mcO (
N \log ( N ) )$ group operations.
In Step~\ref{step:TCRndElts}  we construct $B$ random elements and
compute their $M$-th power.
We compute $t_i^{(2^a)}$ by repeated squaring, ensuring that $a-1 \leq \log_2(N)$, thus
Step~\ref{step:TCReduceTwoPart} can be performed in $B \cdot \log_2(N)$ group operations.
Step~\ref{step:TCNonComm} requires $B\cdot C$ random elements and
$\mcO( B \cdot C )$
group operations;
likewise, Step~\ref{step:TCReturn} requires $\mcO( B \cdot T )$ group
operations.
Thus, the total runtime of the algorithm is $\mcO ( N \log(N)^2
\log(\veps^{-1})^2 ( \mu + \rho ) )$.

Clearly, we only need to store
$\mcO ( \log( N ) \log( \veps^{-1} )^2 )$ elements overall, concluding the proof.
\details{
Cost for step
\begin{itemize}
\item \ref{step:TCRndElts}: $\mcO ( N \log(N)^2 \log( \veps^{-1} ) \mu +  \log(N)\log(\veps^{-1}) \rho )$
\item \ref{step:TCReduceTwoPart}: $\mcO ( \log(N)^2\log\veps^{-1} \mu )$
\item \ref{step:TCNonComm}: $\mcO ( N \log( N ) \log(\veps^{-1})^2 ( \mu + \rho ) )$
\item \ref{step:TCReturn}: $\mcO ( \log( N ) \log( \veps^{-1} )^2 \mu )$
\end{itemize}
}
\end{proof}

\subsection{Construction of a matching cycle} \label{sec:lc}

The aim of this section is, given a $3$-cycle $c$ in a black-box group
$G$ isomorphic to an alternating or symmetric group of degree $n$,
to construct a $k$-cycle~$g$ matching $c$ with $k \ge 3n/4.$
The proportion of cycles with this property
is too small for our purposes, so we consider
other types of elements in $G$ which occur more
frequently and allow the  construction of a $k$-cycle $g$ with the
desired properties. As a first step, we describe what we call bolstering
elements. These allow us to construct the desired cycle $g$ easily.
Since bolstering elements are still too rare, we consider pre-bolstering
elements from which we obtain bolstering elements in turn.

\subsubsection{Bolstering Elements} \label{sec:bolster}

Let $c := (u,v,w)$ be a $3$-cycle.
Call an element $x \in \Sym_n$ \emph{bolstering} with respect to
$c$ if it is of the form
$x = (v,a_1,\ldots,a_{\alpha})(w,b_1,\ldots,b_{\beta})(\ldots)$
or
$x = (v,a_1,\ldots,a_{\alpha},w,b_1,\ldots,b_{\beta})(\ldots)$
with $u \in \fix x$ and $\alpha,\beta \ge 2$.

\begin{remark} \label{rem:BE2LC}
Given a bolstering element $x$ with respect to $c$,
we can find a cycle $g$ matching $c$.
Let $m := \min\{\alpha,\beta\}$ and
$m' := \lfloor |\alpha - \beta|/2 \rfloor$.
\begin{enumerate}
\item $c \cdot c^x \cdot c^{(x^2)} \cdots c^{(x^m)} =: y$ is a single
  cycle of length $2m+3$.
\item If $\alpha \le \beta - 2$, we can compute $z =
  (u,b_{\alpha+2},b_{\alpha+1})$.
Then
$m'$ is the least positive integer such that $z^{(x^{2m'})}c$
does not have order $5$ and
$y \cdot z \cdot z^{(x^2)} \cdot z^{(x^4)} \cdots
z^{(x^{2(m'-1)})} =: g$ is a cycle of length $2m'+2m+3$.
\item If $\beta \le \alpha - 2$, we compute $z := (u,a_{\beta+1},a_{\beta+2})$
to obtain a $(2m'+2m+3)$-cycle in similar fashion.
\end{enumerate}
The details of how to compute $z$ will be described in Algorithm~\algBC.
\end{remark}

Since the proportion of bolstering elements with respect to a given
$3$-cycle
in $\Alt_n$ and $\Sym_n$
is too small, we instead try to find pre-bolstering elements and use
these to construct bolstering elements.

An element $r$ is called \emph{pre-bolstering} with respect to $c$ if it is of the form
\[
    r = (w,u,a_1,\ldots,a_{\alpha})(v,b_1,\ldots,b_{\beta})(\ldots)
\]
or
\[
    r = (w,u,a_1,\ldots,a_{\alpha},v,b_1,\ldots,b_{\beta})(\ldots)
\]
with $\supp c = \{u,v,w\}$ and $\alpha,\beta \ge 2$.
Note that if $r$ is pre-bolstering, then either $x=cr$ or $x=c^2r$ is
bolstering with respect to $c$.

The next lemma gives a criterion when an element $r\in \Sym_n$ is
pre-bolstering with respect to a 3-cycle $c$.

\begin{lemma} \label{la:PreBols}
Let $c\in \Sym_n$ be a $3$-cycle. Then $r$ is pre-bolstering with respect
to $c$ if and only if
$[c^r,c] \neq 1_G$, $c^{(r^2)} \not\in \{c,c^2\}$ and
$[c,c^{(r^2)}] = 1_G$.
\end{lemma}
\begin{proof}
Clearly, if $r$ is pre-bolstering, then the conditions hold.
Conversely, suppose that $r$ is not pre-bolstering.
Then either $\supp c^r \cap \supp c = \emptyset$ or
$\fix r \cap \supp c \neq \emptyset$ or $\min\{ \alpha, \beta \} < 2$.
In the first case we find $[c^r,c]=1_G$.
In both the second and the third case, clearly $\supp c \cap \supp
c^{(r^2)} \neq \emptyset$,
thus either $[c,c^{(r^2)}] \neq 1_G$ or $\supp c = \supp c^{(r^2)}$ hold.
(Note that if the supports of $c$ and $c^{(r^2)}$ coincide, then
$c^{(r^2)} = c$ or $c^{(r^2)} = c^2$.)
\end{proof}

For a group $G$ isomorphic to an alternating or a symmetric group
and a $3$-cycle $c\in G$, the following algorithm constructs a list of
bolstering elements with respect to $c$. It achieves this by selecting
a number of random elements from $G$ and using the criteria in
Lemma~\ref{la:PreBols} to recognise pre-bolstering elements among
these. From these it then constructs bolstering elements with respect
to $c$.

\begin{algorithm}[\algBE] \label{alg:BElts}
\textit{Input:} A group $G$, an element $c \in G$, a real number
$\veps$ with $0 < \veps < 1$ and $N \in \mathbb{N}$.\\
\textit{Output:} A list $B$ with $B \subset G$.\\
\textit{Algorithm:}
\begin{enumerate}
    \item
    Let $S := 7N\lceil \frac{7}{4}\log \veps^{-1} \rceil$ and
    $R := \lceil \frac{7}{4}\log \veps^{-1} \rceil$.

    \item
    Set $C := \emptyset$.
    Repeat the following step at most $S$ times:
    choose a random element $r \in G$;
    if $[c^r,c] \neq 1_G$, $c^{(r^2)} \not\in \{c,c^2\}$,
    $[c,c^{(r^2)}] = 1_G$ and $|C| < R$, then add $r$ to $C$.

    \item \label{step:DecideTCForm}
    For each $r \in C$, compute $z_r := c^{ r c r } c^{ r c^{(r^2)} c
    }$. If $(z_r)^3 = 1_G$, then add $c^2r$ to $B$.
    Otherwise add $cr$ to $B$. Return $B$.
\end{enumerate}
\end{algorithm}

\begin{lemma} \label{la:BEOutput}
Let $7 \le n \le N$, $G \in \{\Sym_n,\Alt_n\}$, $c \in G$ a $3$-cycle
and $0 < \veps < 1$.  Let $B := \algBE(G,c,\veps,N)$. Then $B$ is a
list of random bolstering elements and, with probability at least $1 -
\veps$, we have $|B| \ge \lceil \frac{7}{4}\log \veps^{-1} \rceil$.
\end{lemma}
\begin{proof}
Let $\supp c = \{u,v,w\}$. Clearly, using Lemma~\ref{la:PreBols} the
elements $r$ constructed in Step~1 of
Algorithm~\ref{alg:BElts} are pre-bolstering with respect to~$c$.
Step~\ref{step:DecideTCForm} has to decide
whether $c = (u,v,w)$ or $c = (u,w,v)$.
In the first case $z_r$ is a $3$-cycle, while in the second case $z_r$
is a $5$-cycle.
\details{First case: $z_r = c^{rcr}c^{c^{(r^2)}c} =
  (b_1,a_2,b_2)(w,a_2,b_1) = (w,a_2,b_2)$.
Second case: $z_r = (v,b_2,a_2)(u,b_1,b_2) = (v,u,b_1,b_2,a_2)$.}
Thus, $(z_r)^3 = 1_G$ if and only if $c = (u,v,w)$ and $B$ is a list
of bolstering elements.
By Proposition~\ref{prop:LCGoodElements}, we find less than $R$ elements
with probability at most~$\veps$, since
$S = 7N\lceil \frac{7}{4} \log \veps^{-1} \rceil
\ge 5N \cdot \max\left( \left(5/4\right)^4 \log \veps^{-1},
	\frac{25}{18} \left\lceil \frac{1}{2}
	\log_{3/4} \veps \right\rceil \right)$.
\details{
$S = 7N\lceil \frac{7}{4} \log \veps^{-1} \rceil
\ge 5N\cdot\frac{625}{448} \lceil \frac{7}{4} \log \veps^{-1} \rceil$
and
\begin{align*}
	\frac{625}{448} \left\lceil \frac{7}{4} \log \veps^{-1} \right\rceil
	&\ge \max\left( \frac{5^4}{4^3\cdot7} \cdot \frac{7}{4} \log \veps^{-1},
	\frac{25}{18} \left\lceil \frac{7}{4} \log \veps^{-1} \right\rceil \right)\\
	&\ge \max\left( \left(\frac{5}{4}\right)^4 \log \veps^{-1},
	\frac{25}{18} \left\lceil \frac{1}{2}
	\log_{\frac{3}{4}} \veps \right\rceil \right).
\end{align*}
}
\end{proof}

\begin{lemma} \label{la:BERuntime}
Let $G$ be a black-box group, $c \in G$ an arbitrary element, $0 <
\veps < 1$ and $N \in \mathbb{N}$.
Then algorithm $\algBE$ with input $G,c,\veps,N$ runs in $\mcO(N \log \veps^{-1} ( \mu + \rho ) )$ time
and requires storage of $\mcO( \log \veps^{-1} )$ group elements.
\end{lemma}
\begin{proof}
This is immediate.
\details{Choosing up to $S$ elements and deciding whether they are
pre-bolstering requires $\mcO( N \log \veps^{-1} )$ group operations
and random elements.  Constructing bolstering elements for up to $R$
elements has a cost of $\mcO( \log \veps^{-1} )$ group
operations, concluding the proof.}
\end{proof}

\subsubsection{Exploiting bolstering elements} \label{sec:exploit}

Given a bolstering element $x$ with respect to a $3$-cycle $c$, we can
construct a cycle $g_x$ matching $c$,
using Remark~\ref{rem:BE2LC}.
But depending on the type of the bolstering element, this may require
different steps to obtain the longest possible matching cycle.
The type of a given bolstering element
can be determined using only black-box operations  as described in
Remark~\ref{rem:BolsteringProperties}.
We first describe
Algorithm~\algBC
which applies
this remark to obtain a cycle $g_x$ matching $c$ from a given bolstering
element $x$.  
This is used by Algorithm~\algLC, which computes $g_x$ for every $x$
returned by Algorithm~\algBE, and returns the longest~$g_x$.

\begin{remark} \label{rem:BolsteringProperties}
Several properties of bolstering elements can be checked
algorithmically using only black-box operations.
Let $c = (u,v,w)$. Let $x$ be bolstering with respect to $c$
and $u \in \fix x$.
\begin{enumerate}
\item Let $m := \min\{ \alpha, \beta \}$. Then $m$ is the least
  natural  number such that $c^{(x^{m+1})}c$ does not have order $5$.
Note that necessarily $m < n/2$.
\item $\alpha = \beta$ if and only if $c^{(x^{m+1})} \in \{c,c^2\}$.
\item $|\alpha - \beta| = 1$ if and only if $c^{(x^{m+2})}c$ does not
  have order $5$.
\item If $\alpha \neq \beta$, then $w \not\in v^{\langle x \rangle}$, i.e. $x$ is of the first form, if and only if
$c^{(x^{m+1})}c$ has order $2$.
\item Assume $|\alpha - \beta| > 1$. If $w \in v^{\langle x \rangle}$, then $\alpha > \beta$ if and only if
$c^{(x^{m+2})}$ and $c^{(x^{m+1}c)}$ commute. If $w \not\in v^{\langle x \rangle}$, then $\alpha < \beta$ if and only if
$c^{(x^{m+2})}$ and $c^{(x^{m+1}c)}$ commute.
\end{enumerate}
\end{remark}

When called with input a black-box group $G$ isomorphic to an
alternating group $\Alt_n$ or a symmetric group $\Sym_n$ and elements $c,
x\in G$ such that $c$ is a 3-cycle and $x$ is a bolstering element
with respect to $c$, the following algorithm determines a cycle $g_x$
of length $k$ matching $c$. It returns $g_x$ and its length $k$.

\begin{algorithm}[\algBC] \label{alg:BuildCycle}
\textit{Input:} Elements $c,x$ of a group $G$ and $N \in \mathbb{N}$.\\
\textit{Output:} A number $k \in \mathbb{N}$ and an element $g \in G$, or $\textsc{fail}$.\\
\textit{Algorithm:}
 Determine $m := \min \{\alpha,\beta\}$ and check whether $| \alpha -
 \beta | \ge 2$ as described
in Remark~\ref{rem:BolsteringProperties}. If $m \ge N/2$,
return \textsc{fail}. Compute $y := c \cdot c^x \cdot c^{(x^2)}
\cdots c^{(x^m)}$.
If $|\alpha-\beta| \le 1$, return $2m+3, y$.
Otherwise set $d := c^{(x^{m+1})}$ and
\[
	e := \begin{cases}
		d^{xc}, & \text{ if } w \in v^{\langle x \rangle} \text{ and } \alpha > \beta, \\
		(d^{(xc^2)})^2, & \text{ if } w \in v^{\langle x \rangle} \text{ and } \alpha < \beta, \\
		d^{(xc^2)}, & \text{ if } w \not\in v^{\langle x \rangle} \text{ and } \alpha > \beta, \\
		(d^{(xc)})^2, & \text{ if } w \not\in v^{\langle x \rangle} \text{ and } \alpha < \beta,
	\end{cases}
\]
where we can decide whether $w \in v^{\langle x \rangle}$ and $\alpha > \beta$
using Remark~\ref{rem:BolsteringProperties}.
Set $z := d^e$ and
determine $m'$ as described in Remark~\ref{rem:BE2LC}.
If $m' \ge N/2$, return \textsc{fail}.
Otherwise compute
$g := y \cdot z \cdot z^{(x^2)} \cdots z^{(x^{2(m'-1)})}$.
Return $2m'+2m+3, g$.
\end{algorithm}

\begin{lemma} \label{la:BCOutput}
Let $7 \le n \le N \in \mathbb{N}$,
$c \in \Sym_n$ a $3$-cycle and $x$ a bolstering element with respect to~$c$.
Then $\algBC$ with input $c,x,N$ returns $k$ and $g$ such that $g$ is a $k$-cycle matching $c$.
\end{lemma}
\begin{proof}
This is an application of Remarks~\ref{rem:BE2LC} and~\ref{rem:BolsteringProperties},
where it is easy to check that $z$ has the form given in
Remark~\ref{rem:BE2LC}, e.g., if
$w \in v^{\langle x \rangle}$ and $\alpha > \beta$ we have $d = (u,a_{\beta+1},v)$ and $e = (v,a_{\beta+2},a_1)$,
hence $z = (u,a_{\beta+1},a_{\beta+2})$.
\details{
	We have
	\begin{itemize}
	\item $w \in v^{\langle x \rangle}$, $\alpha > \beta$: $d =
          (u,a_{\beta+1},v)$, $e = (v,a_{\beta+2},a_1)$
	\item $w \in v^{\langle x \rangle}$, $\alpha < \beta$: $d =
          (u,v,b_{\alpha+1})$, $e = (w,b_{\alpha+2},b_1)$
	\item $w \not\in v^{\langle x \rangle}$, $\alpha > \beta$: $d
          = (u,a_{\beta+1},w)$, $e = (w,a_{\beta+2},b_1)$
	\item $w \not\in v^{\langle x \rangle}$, $\alpha < \beta$: $d
          = (u,v,b_{\alpha+1})$, $e = (v,b_{\alpha+2},a_1)$,
	\end{itemize}
	so $z = (u,a_{\beta+1},a_{\beta+2})$ if $\alpha > \beta$ and
        $z = (u,b_{\alpha+2},b_{\alpha+1})$ if $\beta > \alpha$.
}
\end{proof}

\begin{lemma} \label{la:BCRuntime}
Let $G$ be a finite group, $c,x \in G$ arbitrary elements, and $N \in \mathbb{N}$.
Then \algBC with input $c,x,N$ runs in $\mcO ( N \mu ) $ time and requires storage
of a constant number of group elements.
\end{lemma}
\begin{proof}
By storing $c^{(x^{i-1})}$, the next element $c^{(x^i)}$ can be computed in constant time.
Since $m$ and $m'$ are bounded by $N/2$, the lemma follows.
\end{proof}

\begin{algorithm}[\algLC] \label{alg:LC}
\textit{Input:} A group $G$, an element $c \in G$, $0 < \veps < 1$ and $N \in \mathbb{N}$.\\
\textit{Output:} A number $k \in \mathbb{N}$ and an element $g \in G$ or \textsc{fail}.\\
\textit{Algorithm:}
\begin{enumerate}
\item \label{step:LCInit}
Let $B := \algBE(G,c,\veps/2,N)$. If $B$ contains less than
$\lceil \frac{7}{4}\log ( 2/\veps ) \rceil$ elements, return \textsc{fail}.
\item \label{step:LCFirstLength}
Call \algBC for each bolstering element $x \in B$.
If this fails for some $x$, return \textsc{fail}.
Otherwise return $k$ and $g$ computed by \algBC with maximal $k$.
\end{enumerate}
\end{algorithm}

\begin{lemma} \label{la:LCOutput}
Let $9 \le N \in \mathbb{N}$, $0 < \veps < 1$,
$G \in \{ \Sym_n, \Alt_n \}$ for some $9 \le n \le N$ and
$c \in G$ a $3$-cycle. Then, with probability at least $1-\veps$, $\algLC$
with input $G,c,\veps,N$ returns $k$ and $g$ such that
$k \ge \max(3n/4, 9)$ and $g$ is a $k$-cycle matching $c$.
\end{lemma}
\begin{proof}
Step~\ref{step:LCInit} succeeds with probability at least $1 - \veps/2$, cf.~Lemma~\ref{la:BEOutput}.
Since $7/4 \log(2/\veps) \geq 1/2\log_{3/4}(\veps/2)$,
Proposition~\ref{prop:LCBigOrbits} yields that, with probability at least $1-\veps/2$,
\algBC constructs at least one $k$-cycle with $k \geq \max(3n/4,9)$.
\end{proof}

\begin{lemma} \label{la:LCRuntime}
Let $G$ be a finite group, $c \in G$ an arbitrary element, $0 < \veps < 1$ and $N \in \mathbb{N}$.
Then \algLC with input $G,c,\veps,N$
runs in $\mcO ( N \log \veps^{-1} ( \mu + \rho ) ) $ time and requires storage
of $\mcO(\log \veps^{-1})$ group elements.
\end{lemma}
\begin{proof}
This follows from Lemmas~\ref{la:BERuntime} and~\ref{la:BCRuntime}.
\end{proof}

\subsection{Auxiliary algorithms} \label{sec:aux}

In this section we describe short algorithms which are called by
Algorithm~\algGens.
For our discussion, we assume we are given a group $G$
isomorphic to $\Alt_n$ or $\Sym_n$ and
that $c$ is a $3$-cycle and $g$ a  $k$-cycle matching~$c$. 
We  perform computations mainly in $\langle g, c \rangle \cong
\Alt_{k}$.

The first algorithm decides whether a point $i \in \supp g$ is fixed by a
given element $r \in G$.

\begin{remark}\label{rem:DFPSets}
Let $a_1, \dotsc, a_7 \in \N$ be pairwise distinct and
\[
    A := \{\{1,2,i\}\,:\, 3 \leq i \leq 6\}.
\]
If the sets $\{a_1, a_2, a_3\}$, $\{a_1, a_4, a_5\}$, $\{a_1, a_6, a_7\}$
intersect each set in $A$ non-trivially, then $a_1 \in \{1,2\}$.
\end{remark}

This observation allows us to recognise a fixed point of an arbitrary
element $r \in G$ by examining the intersection of the supports
of some aptly chosen elements. If $c$ is a $3$-cycle and $g$ a
matching cycle,  the following algorithm decides whether the single point in
the intersection of the supports
of $c$ and $c^{(g^2)}$ is fixed by $r$.

\begin{algorithm}[\algDFP] \label{alg:DFP}
\textit{Input:} Elements $g,c,r$ of a group $G$.\\
\textit{Output:} \textsc{true} or \textsc{false}.\\
\textit{Algorithm:}
Define
\[
    X := \{ c^r, c^{g^2 r}, c^{g^2 c^{(g^3)} c^{(g^4)} r} \}
\]
and
\[
    H_1 := \{ c^2, c^{c^g}, c^{c^g c^{(g^3)}}, c^{c^g (c^{(g^3)})^2}, c^{c^g (c^{(g^3)})^2 c^{(g^4)}} \}.
\]
If there is an element $x \in X$ such that $[x,h] = 1_G$ for at least two different $h \in H_1$,
then return $\textsc{false}$. Otherwise define
\[
    H_2 := \{ c, c^g, c^{g c^{(g^3)}}, c^{g (c^{(g^3)})^2}, c^{g (c^{(g^3)})^2 c^{(g^4)}} \}.
\]
If there is an element $x \in X$ such that $[x,h] = 1_G$ for at least two different $h \in H_2$,
then return $\textsc{false}$. Otherwise return \textsc{true}.
\end{algorithm}


\begin{lemma} \label{la:DFPOutput}
Let $7 \le k \le n$, $c \in \Sym_n$ a $3$-cycle, $g \in \Sym_n$ a $k$-cycle
matching~$c$ and $r \in \Sym_n$ an arbitrary element.
$\textsc{IsFixedPoint}(g, c, r)$ returns \textsc{true} if and only if the unique point contained in
both $\supp c$ and $\supp c^{(g^2)}$ is fixed by~$r$.
\end{lemma}
\begin{proof}
Without loss of generality, let $c = (1,2,3)$ and $g = (1,2,\ldots,k)$. We find
$\supp c \cap \supp c^{(g^2)} = \{3\}$,
$H_1 = \{ (1,3,j) \,:\, j \in \{2,4,5,6,7\}\}$,
$H_2 = \{ (2,3,j) \,:\, j \in \{1,4,5,6,7\} \}$
and $X = \{ (1,2,3)^r, (3,4,5)^r, (3,6,7)^r \}$.

Assume that $\algDFP$ returns $\textsc{false}$.
Then there are elements $x \in X$ and $h_1, h_2 \in H_1$ (or in $H_2$) commuting with $x$.
Suppose $3 \in \fix r$.
Since then $3 \in \supp x \cap \supp h_1 \cap \supp h_2$ and $h_1,h_2$ commute with $x$,
we obtain $\supp h_1 = \supp x = \supp h_2$, a contradiction. Thus $3 \not\in \fix r$.

Conversely assume that $\algDFP$ returns $\textsc{true}$.
Then, for each $x \in X$, there exist $h_1, \dotsc, h_4 \in H_1$ with $\supp h_i \cap \supp x \neq \emptyset$,
and similarly for $H_2$.
The result now follows by Remark~\ref{rem:DFPSets}.
\end{proof}


\begin{lemma} \label{la:DFPRuntime}
Let $G$ be a finite group and $g,c,r \in G$ arbitrary elements. Then
$\algDFP$ with input $g,c,r$ uses a  constant number of group
operations  and requires storage of a constant number of
group elements.
\end{lemma}
\begin{proof}
This is immediate.
\end{proof}


Let $G$ be a black-box group isomorphic to an alternating or symmetric group,
$c \in G$ a $3$-cycle, $g \in G$ a $k$-cycle matching $c$, and $r$ another element of~$G$.
Assume without loss of generality that $g = (1,2,\ldots,k)$ and $c = (1,2,3)$.
If $r$ satisfies $|\supp r \cap \supp g| \geq 1$ and $|\fix r \cap \supp g| \geq 2$,
the next algorithm computes a conjugate $\wt{r} = r^x$ such that $\wt{r}$ fixes the
points $1$~and~$2$, but not the point $3$.
Here we identify the point $j \in \{1, \dotsc, k\}$ with the $3$-cycle $c^{g^{(j-3)}}$.

\begin{algorithm}[\algFPM] \label{alg:FPM}
\textit{Input:} Elements $g,c,r$ of a group $G$ and $k \in \mathbb{N}$.\\
\textit{Output:} An element $\wt{r} \in G$ conjugate to $r$ or $\textsc{fail}$.\\
\textit{Algorithm:}
Compute the set
\[
	F := \{ 1 \le j \le k : \algDFP(g,c^{(g^{j-3})},r) = \textsc{true} \}.
\]
If $|F| < 2$ or $|F| = k$, then return \textsc{fail}. Otherwise, define $f_1$ as the smallest and $f_2$ as
the second smallest number in $F$. Define $m$ as the smallest natural number not in $F$.
Define the element $x \in G$ according to the following table:
\begin{center}
\begin{tabular}{c|c}
$F \cap \{1,2,3,4\}$ & $x$ \\
\hline
$\{1,2,3,4\}$ or $\{1,2,3\}$ &
$ c^{(gc^2)^{m-3}c}c $ \\
$\{1,2,4\}$ or $\{1,2\}$ &
$ 1_G $ \\
$\{1,3,4\}$ &
$c^g$ \\
$\{1,3\}$ &
$(c^2)^g$ \\
$\{1,4\}$ or $\{1\}$ &
$ c^{(gc^2)^{f_2-3}c} $ \\
$\{2,3,4\}$ or $\{2,4\}$ &
$c^{c^g}$ \\
$\{2,3\}$ &
$(c^2)^{c^g}$ \\
$\{2\}$ &
$c^{(gc^2)^{f_2-3}c^g}$ \\
$\{3,4\}$ or $\{3\}$ &
$(c^2)^{(gc^2)^{f_2-3}}c^2$ \\
$\{4\}$ or $\emptyset$ &
$c^{(gc^2)^{f_2-3}} c^{(gc^2)^{f_1-3}}$
\end{tabular}
\end{center}
Return $\wt{r} := r^x$.
\end{algorithm}

\begin{lemma} \label{la:FPMOutput}
Let $7 \le k_0 \le k \le n \in \mathbb{N}$, $c =(1,2,3)$,
$g =(1,2,\ldots,k)$ and $r \in \Sym_n$ a $k_0$-cycle.
If $r$ has in $\supp g$ at least two fixed points
and one moved point, then $\wt{r} := \algFPM(g,c,r,k)$
is a $k_0$-cycle fixing the points $1$ and $2$ and moving $3$.
Moreover, the difference $\supp r - \supp g$ lies in $\supp \wt{r}$.
\end{lemma}
\begin{proof}
If $r$ has two fixed points and a moved point in $\supp g$, the algorithm returns a $k_0$-cycle $\wt{r}$.
We want to show that $\wt{r}$ fixes the points $1$ and $2$ but moves the point $3$.
By Lemma~\ref{la:DFPOutput}, we have $F = \fix r \cap \supp g$.
Then the table defining $x$ looks as follows:
\begin{center}
\begin{tabular}{c|c}
$F \cap \{1,2,3,4\}$ & $x$ \\
\hline
$\{1,2,3,4\}$ or $\{1,2,3\}$ & $ (1,2)(3,m) $ \\
$\{1,2,4\}$ or $\{1,2\}$ & $ 1_G $ \\
$\{1,3,4\}$ & $ (2,3,4) $ \\
$\{1,3\}$ & $ (2,4,3) $ \\
$\{1,4\}$ or $\{1\}$ & $ (2,3,f_2) $ \\
$\{2,3,4\}$ or $\{2,4\}$ & $ (1,3,4) $ \\
$\{2,3\}$ & $ (1,4,3) $ \\
$\{2\}$ & $ (1,3,f_2) $ \\
$\{3,4\}$ or $\{3\}$ & $ (1,f_2)(2,3) $ \\
$\{4\}$ or $\emptyset$ & $ (1,f_1)(2,f_2) $
\end{tabular}
\end{center}
Thus, in each case $\wt{r} = r^x$ fixes $1$ and $2$ but not $3$.
Since $x \in \langle g, c \rangle$,
it fixes every element in $\{1, \dotsc, n\} - \supp g$, so
$(\supp r - \supp g) \subset \supp \wt{r}$ holds.
\end{proof}

\begin{lemma} \label{la:FPMRuntime}
Let $G$ be a finite group, $g,c,r \in G$ arbitrary elements and $k \in \mathbb{N}$.
$\algFPM$  with input $g,c,r,k$ runs in $\mcO( k \mu )$ time
and requires storage of a constant number of group elements.
\end{lemma}
\begin{proof}
This follows by standard arguments.
\details{%
By storing the previous element $c^{(g^{j-3})}$ during the computation of $F$, the next
element $c^{(g^{j-2})}$ can be computed in a constant number of group operations. Thus, computing $F$ requires
$\mcO( k )$ group operations. Computing the element $x$ requires at most $\mcO( \log k )$
group operations, yielding an overall runtime of $\mcO( k \mu )$.
}
\end{proof}

Using elements provided by $\algFPM$, the next algorithm appends new points to the cycle $g$.
Since $g$ will always be a cycle of odd length, new points can only be appended in pairs.
Because of this we need an element $s$, a `storage cycle', storing the first new point until
we encounter a second one.
The output $\wt{s}$ assumes the role of $s$ the next time $\algEC$ is called.

\begin{algorithm}[\algEC] \label{alg:EC}
\textit{Input:} Elements $g,c,r,s$ of a group $G$ and $k,k_0 \in \mathbb{N}$.\\
\textit{Output:} Two elements $\wt{g},\wt{s} \in G$ and $\wt{k} \in \mathbb{N}$.\\
\textit{Algorithm:}
\begin{enumerate}
\item
Set $\wt{g} := g$, $\wt{s} := s$ and $\wt{k} := k$.

\item
For each $1 \leq j < k_0$, set $x_j := c^{(r^j)}$.
If $[x_j,\wt{g}c^2] = 1_G$, then perform Step~\ref{step:checkPoint}.

\item \label{step:checkPoint}
If $\wt{s} = 1_G$, then set $\wt{s} := x_j$.
If $\wt{s} \neq 1_G$ and $\wt{s} \neq x_j$, then set $\wt{k} := \wt{k}+2$, $\wt{g} := \wt{g}\wt{s}^{(x_j^2)}$ and $\wt{s} := 1_G$.

\item
Return $\wt{g}$, $\wt{s}$ and $\wt{k}$.
\end{enumerate}
\end{algorithm}

\begin{lemma} \label{la:ECOutput}
Let $7 \le k_0 \le k \le n \in \mathbb{N}$, $c = (1,2,3)$, $g = (1,2,\ldots,k)$ and
$r \in \Sym_n$ a $k_0$-cycle fixing the points $1$ and $2$ and moving $3$.
Let $s \in \Sym_n$ be either the identity element or $s = (1,2,b)$ for some $b \in \{ 1, \ldots, n \} - \supp g$.
Let $\wt{g},\wt{s},\wt{k} := \algEC(g,c,r,s,k,k_0)$. Then $\wt{g}$ is a $\wt{k}$-cycle
matching~$c$, and $\supp r \cup \supp g \cup \supp s = \supp \wt{g} \cup \supp \wt{s}$.
\end{lemma}
\begin{proof}
Let $r = (3,a_1,\ldots,a_{k_0-1})$ with $4 \le a_j \le n$. Then $x_j = (1,2,a_j)$,
so $x_j$ and $\wt{g}c^2$ commute if and only if $a_j \not\in \supp \wt{g}$. If, in this case,
$\wt{s}$ is the identity, the new point is stored in $\wt{s}$. If $\wt{s} = x_j$, the point is already stored in $s$. Otherwise
we find $\wt{s} = (1,2,b)$ for some $b \not\in ( \supp \wt{g} \cup \{a_j\} )$. Now, $\wt{g}$ is set to
$(1,2,\ldots,k,b,a_j)$, becoming a $\wt{k}$-cycle matching $c$.
\details{$g := gs^{(x^2)} = (1,2,\ldots,k)(1,b,a_j) = (1,2,\ldots,k,b,a_j)$}
Since all $a_j$ are treated in this manner, clearly
$\supp r \subset ( \supp \wt{g} \cup \supp \wt{s} )$ holds.
\end{proof}


\begin{lemma} \label{la:ECRuntime}
Let $G$ be a finite group, $g,c,r,s \in G$ arbitrary elements and $k,k_0 \in \mathbb{N}$.
Then $\algEC$ with input $g,c,r,s,k,k_0$ runs in $\mcO ( k_0 \mu )$ time and requires
storage of a constant number of group elements.
\end{lemma}
\begin{proof}
This is immediate.
\end{proof}

\subsection{Construction of standard generators} \label{sec:mainalg}

Let $G$ be a black-box group isomorphic to an alternating or symmetric group,
$c \in G$ a $3$-cycle and $g \in G$ a $k$-cycle matching $c$.
The first algorithm in this section uses these elements to construct standard generators
of the alternating group of the same degree as $G$.

The main algorithm \algMain ties up all algorithms in this chapter and results of \cite{Bealsetal}
to either constructively
recognise the group or decide that it is not isomorphic to an alternating or symmetric group
with high probability.

\begin{algorithm}[\algGens] \label{alg:Gens}
\textit{Input:} A group $G$, elements $g,c \in G$, $0<\veps<1$ and $k,N \in \mathbb{N}$.\\
\textit{Output:} Elements $\wt{g},\wt{c} \in G$ and $\wt{k} \in \mathbb{N}$ or \textsc{fail}.\\
\textit{Algorithm:}
\begin{enumerate}
\item \label{step:GensInit}
Set $s := 1_G$,
$k_0 := k - 2$,
$r := gc^2$,
$\wt{k} := k$
and $\wt{g} := g$.
\item \label{step:GensConjs}
Choose a list $R$ of $\lceil \log (10/3)^{-1} ( \log N + \log \veps^{-1} ) \rceil$
random conjugates of~$r$.
For each $x \in R$, perform Step~\ref{step:adjustAppend}.
\item \label{step:adjustAppend}
Set $m := \algFPM(\wt{g},c,x,\wt{k})$. If $m = \textsc{fail}$, then return $\textsc{fail}$.\\
Set $\wt{g},s,\wt{k} := \algEC(\wt{g},c,m,s,\wt{k},k_0)$. If $\wt{k} > N$, then return $\textsc{fail}$.
\item \label{step:GensStd}
If $s = 1_G$, set $\wt{g} := c^2\wt{g}$ and $\wt{c} := c$. Otherwise set $\wt{k} := \wt{k} + 1$, $\wt{g} := \wt{g}s$ and $\wt{c} := s$.
\item \label{step:GensPres}
Check whether $(\wt{g},\wt{c})$ satisfies the presentation~\eqref{eq:PresAltEven} or~\eqref{eq:PresAltOdd} for $\Alt_k$.
If that is not the case, then return \textsc{fail}. Otherwise return $\wt{g},\wt{c},\wt{k}$.
\end{enumerate}
\end{algorithm}

\begin{lemma} \label{la:GensOutput}
Let $9 \le k \le n \le N \in \mathbb{N}$, $k \ge 3n/4$, $G \in \{ \Sym_n, \Alt_n \}$, $c \in G$ a $3$-cycle,
$g \in G$ a $k$-cycle matching $c$ and $0 < \veps < 1$.
Then, with probability at least $1-\veps$, we find
$\wt{g},\wt{c},\wt{k} := \algGens(G,g,c,\veps,k,N) \neq \textsc{fail}$ such that
$\wt{k} = n$ and $\wt{g},\wt{c}$ are standard generators for $\Alt_n$.
\end{lemma}
\begin{proof}
First note that $k_0 \ge \lceil (7/10) n\rceil$ and $r$ is a $k_0$-cycle, so
the supports of $\wt{g}$ and a random conjugate $x$ of $r$ always have a common moved point.
Furthermore, $x$ has at least two fixed points in $\supp \wt{g}$
since $k = k_0+2$,
so the algorithm cannot fail in Step~\ref{step:adjustAppend}.
Lemmas~\ref{la:FPMOutput} and~\ref{la:ECOutput}
ensure that after Step~\ref{step:GensConjs} the set $\supp \wt{g} \cup \supp s$ contains the supports of all $x \in R$.
Thus, by Theorem~\ref{thm:CommonFPs}, we find that with probability at least $1-\veps$ the elements $\wt{g}$ and $s$ have no common
fixed point on $\{1,\ldots,n\}$. It is easy to check that we return the correct degree and standard generators.
\end{proof}

\begin{lemma} \label{la:GensRuntime}
Let $G$ be a group, $g,c \in G$ arbitrary elements, $0 < \veps < 1$ and $k,N \in \mathbb{N}$. Then $\algGens$
with input $G,g,c,\veps,k,N$ runs in $\mcO\left( N (\log N + \log \veps^{-1} ) (\mu + \rho) \right)$ time and requires storage
of a constant number of group elements.
\end{lemma}
\begin{proof}
The cost to check whether a presentation for $\Alt_k$ is satisfied requires $\mcO( N )$ group
operations by \cite[Lemma 4.4]{Bealsetal}.
At any call of \algFPM and \algEC we have $k \le N$.
Thus, Lemmas~\ref{la:FPMRuntime} and~\ref{la:ECRuntime} yield the claimed runtime.
\end{proof}

We can now present the main algorithm and prove the main Theorem~\ref{thm:main}.

\begin{algorithm}[\algMain] \label{alg:Main}
\textit{Input:} A group $G = \langle X \rangle$, $0<\veps<1$ and $N \in \mathbb{N}$.\\
\textit{Output:} A constructive isomorphism or $\textsc{fail}$.\\
\textit{Algorithm:}
\begin{enumerate}
\item \label{step:MainInit}
Set $T := \lceil \log_2 \veps^{-1} \rceil$.
\item \label{step:MainTCs}
If $T = 0$, then return \textsc{fail}. Otherwise set $T := T-1$ and compute $R := \algTC(G,1/4,N)$.
If $R = \textsc{fail}$, then return \textsc{fail}.
\item \label{step:MainNextCand}
If $R = \emptyset$, go to Step~\ref{step:MainTCs}. Otherwise choose $c \in R$ and set $R := R - \{c\}$.
\item \label{step:MainLC}
Set $\ell := \algLC(G,c,1/8,N)$. If $\ell = \textsc{fail}$, go to Step~\ref{step:MainNextCand}.
Otherwise set $k,g := \ell \in \mathbb{N} \times G$.
\item \label{step:MainGens}
Set $\ell := \algGens(G,g,c,1/8,k,N)$. If $\ell = \textsc{fail}$, go to Step~\ref{step:MainNextCand}.
Otherwise set $g,c,n := \ell \in G \times G \times \mathbb{N}$.
\item \label{step:MainCheck}
Using methods described in \cite{Bealsetal}, check whether $G$ is isomorphic to $\Alt_n$ or $\Sym_n$. If that is the case,
then return the constructive isomorphism computed during the check. Otherwise go to Step~\ref{step:MainNextCand}.
\end{enumerate}
\end{algorithm}

\begin{proof}[Proof of Theorem~\ref{thm:main}]
For the first part of the statement, consider Steps~\ref{step:MainTCs} through~\ref{step:MainCheck}.
Note that \algTC cannot fail if $G$ is an alternating or symmetric group of degree at most~$N$,
so by Lemma~\ref{la:TCOutput} we obtain a set $R$ containing a $3$-cycle with probability at least $3/4$. Thus, without loss of
generality, let $c \in R$ be a $3$-cycle. Using Lemma~\ref{la:LCOutput}, we find, with probability at least $7/8$,
that Step~\ref{step:MainLC} constructs a $k$-cycle matching $c$ with $k \ge \max(3n/4, 9)$.
Now, by Lemma~\ref{la:GensOutput}, Step~\ref{step:MainGens} returns the correct degree and standard generators with
probability at least $7/8$.
Step~\ref{step:MainCheck} always returns a correct answer, cf. \cite[Lemma 5.5 and proof of Theorem 1.2(b)]{Bealsetal}.
Thus, the probability to succeed in one pass is at least $(3/4)\cdot(7/8)^2 > 1/2$.
We repeat this procedure $\lceil \log_2 \veps^{-1} \rceil$ times to obtain the claimed overall probability.
\details{
\[
	\prob( \textsc{fail} ) < (1/2)^{ \lceil \log_2 \veps^{-1} \rceil } \le 2^{-\log_2 \veps^{-1} } = \veps.
\]
}

We now prove the second claim.
Steps~\ref{step:MainTCs} through~\ref{step:MainCheck} are repeated up to $\lceil \log_2 \veps^{-1} \rceil$ times.
During one such pass we execute Step~\ref{step:MainTCs} only once and Steps~\ref{step:MainLC}
through~\ref{step:MainCheck} up to $|R|$ times.
By Lemma~\ref{la:TCOutput} we have $|R| \le c \log N$ for some constant $c \in \mathbb{R}$.
In Step~\ref{step:MainGens}, note that $k,n \le N$ must hold. Then the claim follows by Lemmas~\ref{la:TCRuntime},
\ref{la:LCRuntime}, \ref{la:GensRuntime} and~\cite[Section 5]{Bealsetal}.
\end{proof}

\section{Probability estimates}\label{sec:techno}

This section contains theoretical results which are used to
establish lower bounds for the success probability of 
the algorithm.
Several results are of independent interest. We already mentioned
the probability estimates for small support involutions in the introduction.
Another noteworthy result is a lower bound on the proportion on $k$-cycles
in $\Sym_n$ having a common fixed point, cf.~Theorem~\ref{thm:CommonFPs}.

Note that if $f$ is a continuous and decreasing function on the
interval $[a,b+1]$, then
\begin{equation}\label{eq:int1}
\int_a^{b+1} f(x) \dint{x} \le \sum_{k=a}^b f(k).
\end{equation}

We will also use the following useful result several times.
\begin{lemma}[Chernoff's bound, {\cite[Lemma 2.3.3]{akos}}]
\label{la:chernoff}
Let $X_1, X_2, \ldots$ be a sequence of $0$-$1$ valued random variables
such that $\prob(X_i = 1) \ge p$ for any values of the previous $X_j$
(but $X_i$ may depend on these $X_j$). Then, for all integers $T$ and $0 < \delta < 1$,
\[
    \prob\left(\sum_{i = 1}^T X_i \le (1-\delta)pT\right) \le e^{-\delta^2pT/2}.
\]
\end{lemma}

\subsection{Small support involutions}

The aim of this section is to compute the proportion of even-order
elements in $\Alt_n$ and $\Sym_n$ which power to an involution with small support.
These involutions are used in the algorithm to construct $3$-cycles
(cf.~Algorithm~\ref{alg:tc} and Corollary~\ref{cor:TCgivenNC}).
To achieve this, we compute lower bounds for the proportion $u_{b}(n)$
of elements in $\Sym_n$ and the proportion $\wt{u}_b(n)$ of elements
in $\Alt_n$ which contain $jb$ points
in cycles of lengths divisible by $b$
but not by $2b$ and the remaining $(n-jb)$ points
in cycles of length not divisible by $b$ for some integer $j$
satisfying
$1 \le j \le
4\sqrt{n}/(3b)
$.
To obtain involutions, we choose $b$ to be a certain power of two.

Let $t_b(bn)$ denote the proportion of all permutations in $\Sym_{bn}$ such
that all cycle lengths are a multiple of $b$ but no cycle length is a
multiple of $2b$. Define $t_b(0) := 1$.
Observe that $t_b(b) = 1/b,$ since the only allowable
permutations are the  $b$-cycles and the
proportion of $b$-cycles in $\Sym_b$ is $1/b$.
The proof of the following lemma refines the ideas in \cite{NPPY} to obtain the
explicit lower bound given below.

\begin{lemma} \label{la:tb}
Let $ n, b \in \mathbb{N}$.
Then $t_b(bn) \ge \left(b^2 3^{1/(2b)} n^{1-1/(2b)} \right)^{-1}$.
\end{lemma}
\begin{proof}
The proof is by induction on $n$.
For $n=1$ we have $t_b(b) = 1/b$ and the claim holds.
Consider $t_{b}((n+1)b)$.
If $1$ lies in a cycle of length $jb$,
then $j$ has to be odd. Choosing $jb-1$ out of
$(n+1)b-1$ points and arranging them yields
$\frac{((n+1)b-1)!}{((n-j+1)b)!}$ such cycles.
On the remaining $(n+1-j)b$ points we may choose any permutation
whose cycles have lengths divisible by $b$ but not by $2b$.
We obtain the recursion
\begin{align*}
	((n+1)b)! \cdot t_b((n+1)b)
	\detailsalign{=
	\sum_{\begin{subarray}{c} j=1 \\ j~\text{odd} \end{subarray}}^{n+1}
	\frac{((n+1)b-1)!}{((n-j+1)b)!}  ((n+1-j)b)! \cdot t_{b}((n+1-j)b)}
	&=
	\sum_{\begin{subarray}{c} j=1 \\ j~\text{odd} \end{subarray}}^{n+1}
	((n+1)b-1)! \cdot t_{b}((n+1-j)b),
\end{align*}
and thus
\[
	(n+1)b \cdot t_b((n+1)b) =
	\sum_{\begin{subarray}{c} j=1 \\ j~\text{odd} \end{subarray}}^{n+1}
	t_{b}((n+1-j)b).
\]
Let us first assume that $n$ is even. The induction hypothesis yields
\begin{align*}
	(n+1)b \cdot t_b((n+1)b)
	\detailsalign{=
	1 + \sum_{\begin{subarray}{c} j=1 \\ j~\text{odd} \end{subarray}}^{n-1}
	t_{b}((n+1-j)b)}
	& \ge
	1 + \sum_{\begin{subarray}{c} j=1 \\ j~\text{odd} \end{subarray}}^{n-1}
	\frac{1}{b^23^{1/(2b)}(n+1-j)^{1-1/(2b)}}\\
	& =
	1 + \sum_{k=1}^{n/2} \frac{1}{b^23^{1/(2b)}(2k)^{1-1/(2b)}}\\
	&\ge
	1 + \left(\frac{2}{3}\right)^{1/(2b)}
	\int_{1}^{n/2+1} \frac{1}{b^22x^{1-1/(2b)}} \dint{x}\\
	&=
	1 + \left(\frac{2}{3}\right)^{1/(2b)} \frac{1}{b} \left( x^{1/(2b)}\bigg|_{x=1}^{n/2+1} \right) \\
	\detailsalign{=
	1 + \left(\frac{2}{3}\right)^{1/(2b)} \frac{1}{b}
	\left((n/2+1)^{1/(2b)} - 1 \right) }
	&\ge
	\frac{1}{b3^{1/(2b)}}
	\left((n+2)^{1/(2b)} - 2^{1/(2b)} + b3^{1/(2b)} \right) \\
	&\ge
	\frac{1}{b3^{1/(2b)}}(n+1)^{1/(2b)}.
\end{align*}
\details{
Now suppose $n$ is odd. Since $t_b(b) = 1/b$ we get
\begin{align*}
	(n+1)b \cdot t_b((n+1)b)
	&=
	\frac{1}{b} + \sum_{\begin{subarray}{c} j=1 \\ j~\text{odd} \end{subarray}}^{n-2}
	t_{b}((n+1-j)b)\\
	&\ge
	\frac{1}{b} + \sum_{\begin{subarray}{c} j=1 \\ j~\text{odd} \end{subarray}}^{n-2}
	\frac{1}{b^23^{1/(2b)}(n+1-j)^{1-1/(2b)}}\\
	&=
	\frac{1}{b} + \sum_{k=1}^{(n-1)/2} \frac{1}{b^23^{1/(2b)}(2k+1)^{1-1/(2b)}}\\
	&\ge
	\frac{1}{b} + \frac{1}{b^23^{1/(2b)}}
	\int_{1}^{(n+1)/2} \frac{1}{(2x+1)^{1-1/(2b)}} \dint{x}\\
	&=
	\frac{1}{b} + \frac{1}{b3^{1/(2b)}}
	\left( (2x+1)^{1/(2b)}\bigg|_{x=1}^{(n+1)/2} \right)\\
	&=
	\frac{1}{b} + \frac{1}{b3^{1/(2b)}}
	\left( (n+2)^{1/(2b)} - 3^{1/(2b)} \right) \\
	&\ge
	\frac{1}{b3^{1/(2b)}}(n+1)^{1/(2b)}.
\end{align*}
Therefore we get in either case that
}%
A similar estimation holds for odd $n$, using $t_b(b) = 1/b$;
in either case we see
\[
	t_b((n+1)b) \ge \frac{1}{b^23^{1/(2b)}}(n+1)^{1/(2b)-1},
\]
so the result follows by induction.
\end{proof}

\details{
The next lemma allows us to apply
inequality~\eqref{eq:int1} in the
proof of Lemma~\ref{la:ub}.
}

\begin{lemma}\label{la:fDecrease}
Let
$f(x) := \left( (n/b-x)^{1/b} x^{1-1/(2b)} \right) ^ {-1}$,
where $n\ge404$ and $1 \le b \le 4\sqrt{n}/3$.
Then $f$ is positive and decreasing for $0 < x \le 4\sqrt{n}/(3b)+1$.
\end{lemma}
\begin{proof}
Clearly $f$ is positive on the given interval. %
\details{
For $n \ge 404$ we have $\sqrt{n} > 20 > \frac{4}{3}$ and thus
$x \le \frac{4\sqrt{n}}{3b} + 1 < \frac{n}{b}$. This yields $f(x) > 0$.
}%
Moreover,
\begin{align*}
\diff{x}f(x)
\detailsalign{= \frac{1}{b}\left(\frac{n}{b} - x\right)^{-\frac{1}{b}-1} x^{\frac{1}{2b}-1}
	+ \left( \frac{1}{2b} - 1 \right) \left( \frac{n}{b} - x \right)^{-\frac{1}{b}} x^{\frac{1}{2b}-2}}
	\detailsalign{= \frac{\frac{x}{b}+( \frac{1}{2b} - 1 ) ( \frac{n}{b} - x )}{( \frac{n}{b} - x )x} f(x)}
	&= \frac{bx+n-2bn+2b^2x}{2(n - bx )bx} f(x)
\end{align*}
and $bx+n-2bn+2b^2x < 0$ for $x \leq 4\sqrt{n}/(3b) + 1$, which proves the claim.%
\details{
For $x \le \frac{4\sqrt{n}}{3b}+1$ we get
\begin{align*}
	bx + 2b^2x + n
	&\le \left(\frac{4}{3}+\frac{8}{3}b\right)\sqrt{n} + n + b + 2b^2 \\
	&\le \left( \frac{4+8b}{3\sqrt{404}}+1\right) n + \frac{bn}{404} + \frac{8}{3}b\sqrt{n}\\
	&\le \left( \frac{4+16b}{3\sqrt{404}}+1+\frac{b}{404}\right) n \\
	&< \left( \frac{4+16b}{3\cdot20}+1+\frac{b}{404}\right) n\\
	&\stackrel{b\ge1}{\le} \left(\frac{1}{3}b+b+\frac{1}{404}b\right) n \\
	&< 2bn,
\end{align*}
so $\diff{x}f(x) < 0$ and $f$ is decreasing.
}%
\end{proof}

Let $\vnot{s}{b}(n)$ denote the proportion of elements in $\Sym_n$
with no cycle of length a multiple of $b$.
Applying the inequality from \cite[Theorem 2.3(b)]{BealsetalII}
we get
\begin{equation}
\label{eq:sb}
	\vnot{s}{b}(n)  \ge
	\frac{b^{1/b}}{\Gamma(1-1/b)n^{1/b}}
	\left(1-\frac{1}{n}\right),
\end{equation}
where $\Gamma$ denotes the $\Gamma$-function.
Now we are in a position to prove the
following lemma which is essential
for the proof of Theorem~\ref{thm:SmallSupport}.

\begin{lemma} \label{la:ub}
Let $404 \le n \in \mathbb{N}$. Define $b := 2^{\lceil \log_2(\frac{1}{3}\log(n))\rceil}$.
Then $u_b(n) \ge 1 / \left( 16\log(n) \right)$ and
$u_{2b}(n) \ge 1 / \left( 21\log(n) \right)$.
\end{lemma}
\begin{proof}
Clearly 
\[
	u_b(n) = \sum_{j=1}^{ \lfloor \frac{4\sqrt{n}}{3b} \rfloor }
	\vnot{s}{b}(n-jb) \cdot t_b(jb).
\]
Set $c(b) := \Gamma( 1 - 1/b ) ^{-1} \left(1-(404-\frac{4}{3}\sqrt{404})^{-1}\right)$; then
$\vnot{s}{b}(n-jb) \ge c(b) \cdot \left( b / (n-jb)\right)^{1/b}$.
Together with Lemmas~\ref{la:tb} and~\ref{la:fDecrease} we obtain
\details{
For $n \ge 404$ we also have
\[
	\left(1-\frac{1}{n-jb}\right) \ge
	\left(1-\frac{1}{n-\frac{4}{3}\sqrt{n}}\right) \ge
	\left(1-\frac{1}{404-\frac{4}{3}\sqrt{404}}\right)
\]
and with the definition $c(b) := \frac{1}{\Gamma( 1 - 1/b )}
(1-\frac{1}{404-\frac{4}{3}\sqrt{404}})$
we get
$\vnot{s}{b}(n-jb) \ge
c(b) \cdot \frac{b^{1/b}}{(n-jb)^{1/b}}.$
Together with Lemma~\ref{la:tb} we obtain
\[
	u_{b}(n)
	\ge \frac{c(b)}{3^{1/(2b)}b^2}
	\sum_{j=1}^{ \lfloor \frac{4\sqrt{n}}{3b} \rfloor }
	\frac{1}{(n/b-j)^{1/b}} \frac{1}{j^{1-1/(2b)}}.
\]
Lemma~\ref{la:fDecrease} shows that
$f(j) = \frac{1}{(n/b-j)^{1/b}} \frac{1}{j^{1-1/(2b)}}$ is positive
and decreasing for
$0 < j \le \frac{4\sqrt{n}}{3j}+1$.
Thus we can apply inequality~\eqref{eq:int1} to get
}
\begin{align*}
	u_{b}(n)
	& \ge \frac{c(b)}{3^{1/(2b)}b^2}
	\sum_{j=1}^{ \lfloor \frac{4\sqrt{n}}{3b} \rfloor }
	\frac{1}{(n/b-j)^{1/b}} \frac{1}{j^{1-1/(2b)}} \\
\detailsalign{\ge
\frac{c(b)}{3^{1/(2b)}b^2}
\int_{1}^{\lfloor \frac{ 4\sqrt{n}}{3b} \rfloor + 1}
\frac{1}{(n/b-j)^{1/b}} \frac{1}{j^{1-1/(2b)}} \dint{j}}
&\detailsalt{\stackrel{\textcolor{details}{j > 0}}{\geq}}{\geq}
\frac{c(b)}{3^{1/(2b)}b^2}
\int_{1}^{\lfloor \frac{ 4\sqrt{n}}{3b} \rfloor + 1}
\frac{1}{(n/b)^{1/b}} \frac{1}{j^{1-1/(2b)}} \dint{j}\\
&\ge
\frac{c(b)}{3^{1/(2b)}b^{2-1/b}n^{1/b} }
\int_{1}^{\frac{ 4\sqrt{n}}{3b} }
j^{1/(2b)-1} \dint{j} \\
&=
\frac{2c(b)}{3^{1/(2b)}b^{1-1/b} n^{1/b} }
j^{1/(2b)}
\bigg|_{j=1}^{\frac{4\sqrt{n}}{3b}} \\
\detailsalign{=
\left(\frac{b}{\sqrt{3}}\right)^{1/b}\frac{2c(b)}{ b n^{1/b} }
\left( \left(\frac{4\sqrt{n}}{3b} \right)^{1/(2b)} - 1 \right)}
&\detailsalt{\stackrel{\textcolor{details}{b \ge 4 > \sqrt{3}}}{>}}{>}
\frac{2c(b)}{ b n^{1/b} }
\left( \left(\frac{4\sqrt{n}}{3b} \right)^{1/(2b)} - 1 \right).
\end{align*}
By definition, $b = 2^{\lceil \log_2(\frac{1}{3}\log(n)) \rceil}$,
thus $\frac{1}{3}\log(n) \le b < \frac{2}{3}\log(n)$.
Note that $\frac{1}{3}\log(n) > 2$
for $n \ge 404$ implies $b \ge 4$. %
\details{
Define $r(b) := \frac{1}{b n^{1/b}}$.
Then $ \diff{b} r(b) = \frac{1}{b^3 n^{1/b}} ( \log(n) - b ) $,
so $r(b)$ is increasing for $0 < b < \log(n)$.
}%
Moreover, $1/ \left(b n^{1/b}\right)$ is increasing in $b$ for $0 < b < \log(n)$, and $\Gamma$ is decreasing on the interval $(0,1)$,
so $c(b)$ is increasing for $b>1$.
Lastly $ \left(4\sqrt{n}/(3b) \right)^{1/(2b)} - 1 $
is decreasing in $b$ for
$0 < b \le 4\sqrt{n}/3 $. Altogether we obtain
\begin{align*}
u_b(n) &\ge
\frac{2c(4)}{ b n^{1/b} }
\left( \left(\frac{4\sqrt{n}}{3b} \right)^{1/(2b)} - 1 \right)\\
&\ge
\frac{6c(4)}{ \log(n) n^{3/\log(n)} }
\left( \left(\frac{2\sqrt{n}}{\log(n)} \right)^{3/(4\log(n))} - 1 \right).
\end{align*}
Since $\bigl(2\sqrt{n}/\log(n) \bigr)^{3/(4\log(n))} - 1 $
is increasing on the interval $[404,\infty)$ %
\details{
Let $q(n) := \bigl(\frac{2\sqrt{n}}{\log(n)} \bigr)^{3/(4\log(n))} - 1$
Then
\[
	\diff{n}q(n) =
	\frac{3}{4\log(n)^2n} \left(\frac{2\sqrt{n}}{\log(n)} \right)^{3/(4\log(n))}
	\Biggl( \log\biggl( \frac{\log(n)}{2} \biggr) - 1\Biggr).
\]
Thus $q$ is increasing on the interval $[404,\infty)$.
}%
and $n^{(3/\log(n))} = e^3$,
this yields
\[
	u_b(n) \ge
	\frac{6c(4)}{e^3 \log(n) }
	\left( \left(\frac{2\sqrt{404}}{\log(404)} \right)^{3/(4\log(404))} - 1 \right)
	\ge \frac{1}{ 16 \log(n) }.
\]
A similar argument establishes the bound for $u_{2b}(n)$. %
\details{
To prove the second claim, note that $\frac{2}{3}\log(n) \le 2b < \frac{4}{3}\log(n)$.
As seen above, $r(t) = \frac{1}{t n^{1/t}}$ is increasing for
$0 < t < \log(n)$ and decreasing for $t > \log(n)$. Moreover,
$r\bigl( \frac{2}{3} \log(n) \bigr) - r\bigl( \frac{4}{3} \log(n) \bigr)
= \frac{3}{4e^{3/4} \log(n)} \bigl( \frac{2}{e^{3/4}} - 1 \bigr) < 0$, so
$r(t)$ assumes its minimal value on the interval
$[ \frac{2}{3}\log(n), \frac{4}{3}\log(n) ]$
at $t = \frac{2}{3}\log(n)$.
This yields
\begin{eqnarray*}
u_{2b}(n) &\ge&
\frac{c(8)}{ b n^{1/(2b)} }
\left( \left(\frac{2\sqrt{n}}{3b} \right)^{1/(4b)} - 1 \right)\\
&\ge&
\frac{3c(8)}{ e^{3/2} \log(n) }
\left( \left(\frac{\sqrt{n}}{\log(n)} \right)^{3/(8 \log(n))} - 1 \right).
\end{eqnarray*}
Let $\wt{q}(n) :=
(\frac{\sqrt{n}}{\log(n)} )^{3/(8\log(n))} - 1 $.
Then we have
\[
	\diff{n}\wt{q}(n) =
	\frac{6}{16 \log(n)^2 n} \left( \frac{\sqrt{n}}{\log(n)} \right)^{3/(8\log(n))}
	\biggl( \log\bigl( \log(n) \bigr) - 1\biggr),
\]
so $\wt{q}$ is increasing on $[404,\infty)$ and we get
\[
	u_{2b}(n) \ge
	\frac{3c(8)}{ e^{3/2} \log(n) }
	\left( \left(\frac{\sqrt{404}}{\log(404)} \right)^{3/(8 \log(404))} - 1 \right)
	\ge \frac{1}{ 21 \log(n) },
\]
which concludes the proof.
}%
\end{proof}

\begin{lemma}
\label{la:utb}
For all $b, n \in \N$,
\[
    \wt{u}_b(n) \geq \left(1 - \frac{1}{b-1}\right)u_b(n).
\]
\end{lemma}
\begin{proof}
Denote by $\vnot{a}{b}(n)$ the proportion of elements
in $\Alt_n$ with no cycle of length a multiple of $b$, and by $\vnot{c}{b}(n) = 2\vnot{s}{b}(n) - \vnot{a}{b}(n)$ the proportion of
such elements in $\Sym_n - \Alt_n$.
Every element in $\Sym_{jb}$ can be supplemented with an element of $\Alt_{n-jb}$ or $\Sym_{n-jb} - \Alt_{n-jb}$
to get an element of $\Alt_n$, hence
\[
	\wt{u}_b(n) \ge \sum_{j=1}^{ \lfloor \frac{4\sqrt{n}}{3b} \rfloor }
	\min \{ \vnot{a}{b}(n-jb), \vnot{c}{b}(n-jb) \}
	\cdot t_b( jb).
\]
Using the bounds $(1-1/(b-1))\vnot{s}{b}(n) \le \vnot{a}{b}(n) \le (1+1/(b-1))\vnot{s}{b}(n)$
from \cite[Theorem 3.3(b)]{BealsetalII} we get
$\vnot{c}{b}(n) \ge (1-1/(b-1))\vnot{s}{b}(n)$,
which yields the result.
\end{proof}

Before proving Theorem~\ref{thm:SmallSupport}, we state the following immediate corollary.

\begin{corollary} \label{cor:SmallSuppElts}
Let $9 \le n \in \mathbb{N}$, $G \in \{\Alt_n,\Sym_n\}$ and
$T := \lceil 13 \log n \log \veps^{-1} \rceil$.
The probability that among $T$ random elements of $G$
there is an element $x$ of even order satisfying
$\big| \supp x^{(|x|/2)} \big| \le \bigl\lfloor 4\sqrt{n}/3 \bigr\rfloor$
is at least $1 - \veps$.
\end{corollary}

\begin{proof}[Proof of Theorem~\ref{thm:SmallSupport}]
The proportion in $\Sym_n$ equals
$\sum_{b \in B_n} u_b(n)$
and in $\Alt_n$ it equals
$\sum_{b \in B_n} \wt{u}_b(n)$,
where $B_n := \{ 2^t : 1 \le t \le \bigl\lfloor\log_2 ( \bigl\lfloor 4 \sqrt{n} / 3 \bigr\rfloor )\bigr\rfloor \}$.
First, let $n \ge 404$ and $b_0 := 2^{\lceil \log_2(\frac{1}{3}\log(n))\rceil}$.
Then Lemmas~\ref{la:ub} and~\ref{la:utb} yield
\begin{align*}
	\sum_{b \in B_n} u_b(n) & \ge \sum_{b \in B_n} \wt{u}_b(n)
    \ge \wt{u}_{b_0}(n) + \wt{u}_{2b_0}(n)
	\ge \frac{1}{13\log(n)}.
\end{align*}
\details{%
\begin{align*}
    \wt{u}_{b_0} + \wt{u}_{2b_0}
	& \ge \left( 1 - \frac{1}{4-1} \right) u_{b_0} + \left( 1 - \frac{1}{8-1} \right) u_{2b_0} \\
	& \ge \frac{2}{3\cdot16\log(n)} + \frac{6}{7\cdot21\log(n)} \ge \frac{1}{13\log(n)}.
\end{align*}
}
For $36 \le n \le 403$ we can check
\[
	\sum_{b \in B_n} \widetilde{u}_b(n) \geq \sum_{b \in B_n} \left(1-\frac{1}{b-1}\right)
    \sum_{j=1}^{ \lfloor \frac{4\sqrt{n}}{3b} \rfloor } \vnot{s}{b}(n-jb) \cdot t_b(jb)  \geq \frac{1}{13\log(n)}
\]
case by case, using the bounds in Lemma~\ref{la:tb} and \eqref{eq:sb}. %
\details{
For $36 \le n \le 403$ we can check
\begin{align*}
	r(n) & :=
	\sum_{b \in B_n} \left(
	\frac{1-1/(b-1)}{3^{1/(2b)}b^2 \Gamma(1-1/b)}
	\sum_{j=1}^{ \lfloor \frac{4\sqrt{n}}{3b} \rfloor }
	\frac{1-1/(n-jb)}{(n/b-j)^{1/b}j^{1-1/(2b)}}
	\right) \\
	& \ge \frac{1}{13\log(n)}
\end{align*}
case by case. This proves the claim for $n \ge 36$ since
\[
	\sum_{b \in B_n} u_b(n) \ge \sum_{b \in B_n} \wt{u}_b(n) \ge
	\sum_{b \in B_n} \left(1-\frac{1}{b-1}\right) u_b(n) \ge r(n).
\]
}%
Lastly, note that the desired property depends only on the cycle type.
For $9 \le n \le 35$, we confirm the claim by investigating each
conjugacy class of $\Sym_n$ and $\Alt_n$ and thus directly computing
the exact proportion.
\end{proof}

\subsection{Products of $k$-involutions}

We call a product of $k$ disjoint transpositions a $k$-involution.
Our method to construct a $3$-cycle uses the product of two random $k$-involutions
$r$~and~$s$ such that $\supp(r) \cap \supp(s)$ contains a single element.
Since we are in a black-box setting, given an involution $r$ we know neither $k$
nor $\supp(r)$ explicitly.
However, if $k$ is small enough, then a random conjugate of $r$ which does not commute with $r$
satisfies our hypothesis with high probability, cf.~Theorem~\ref{thm:InvToCyc}.
Furthermore, there are enough non-commuting conjugates of $r$.
Note that we can find involutions with small $k$ by Theorem~\ref{thm:SmallSupport}.

First, we need some auxiliary lemmas.
\begin{lemma} \label{la:MonoPower}
Let $f(k) := \left( 1 - 2k/( 9k^2/4 - 2k + 1) \right)^{2k}$.
Then $f(k)$ is increasing for $k\ge2$.
\end{lemma}
\begin{proof}
Let $g(k) := \frac{9}{4}k^2-2k+1$.
The derivative of $f(k)$ is
\begin{equation*}
	\Biggl(  2 \log \biggl( 1 - \frac{2k}{g(k)} \biggr) +
	\frac{2k}{g(k)^2} \frac{g(k)}{\frac{9}{4}k^2-4k+1}
	\biggl( -2 g(k) + 2k\Bigl(\frac{9}{2}k - 2\Bigr) \biggr) \Biggr)
	f(k).
\end{equation*}
Thus, using $\log(1+x) \ge x/(1+x)$, we find
\begin{align*}
	\diff{k}f(k)
	\ge
	\! \Biggl( \frac{-4k}{\frac{9}{4}k^2-4k+1}
	+ \frac{2k}{(\frac{9}{4}k^2-2k+1)(\frac{9}{4}k^2-4k+1)}
	\biggl( \frac{9}{2}k^2 - 2 \biggr) \Biggr) f(k)
\end{align*}
and for $k \ge 2$ it is easy to check that both factors are positive.
\details{
Furthermore, for $k>1$, we have
\[
	-4k\left(\frac{9}{4}k^2-2k+1\right)+2k 	\left( \frac{9}{2}k^2 - 2 \right)
	= 8k^2-8k >0,
\]
so $\diff{k}f(k) \ge 0$.
}%
\end{proof}

Let $s \in \Sym_n$ be a fixed $k$-involution.
Denote by $\inv (n,k)$ the number of $k$-involutions in $\Sym_n$.
Then
\[
	\inv (n,k) = \frac{|\Sym_n|}{|\Centra_{\Sym_n}(s)|} = \frac{n!}{2^kk!(n-2k)!}.
\]
Let $\trip (n,k)$ denote the proportion of $k$-involutions
$r \in \Sym_n$ such that $r$~and~$s$ move a single common point.

Note that if $k$ is even, then $\inv(n, k)$ is also the number of $k$-involutions in $\Alt_n$,
and $\trip(n,k)$ equals the proportion of $k$-involutions $r \in \Alt_n$ such that $|\supp(r) \cap \supp(s)| = 1$.
Thus for the results in this section it does not matter whether we consider the alternating or the symmetric group.

\begin{lemma} \label{la:TripConcave}
Let $9 \le n \in \mathbb{N}$ and $1 \le k \le 2\sqrt{n}/3$.
Then $\trip (n,k) \ge \min \{ \trip(n,1), \trip(n,\lfloor 2\sqrt{n}/3\rfloor) \}$.
\end{lemma}
\begin{proof}
We have
\[
	\trip(n,k) = \frac{2k(n-2k)\inv(n-2k-1,k-1)}{\inv(n,k)} = \frac{4k^2(n-2k)!^2}{n!(n-4k+1)!}.
\]
It suffices to show that $\trip(n,k+1)/\trip(n,k)$ is decreasing in $k$.
To see this, consider the derivative of the quotient. We find
\[
	\diff{k}\frac{\trip(n,k+1)}{\trip(n,k)} = \alpha(n,k) \frac{-2(k+1)}{((n-2k-1)(n-2k)k)^3}
\]
for some polynomial $\alpha(n,k) \in \mathbb{Z}[n,k]$. Since $(n-2k-1) > 0$ holds for $n\ge9$
and $k \le 2\sqrt{n}/3$, we only need to show that $\alpha(n,k) \ge 0$.
Write $\alpha = \alpha_{+} + \alpha_{-}$ such that
$\alpha_{+}(n,k) \in \mathbb{Z}_{>0}[n,k]$ and
$\alpha_{-}(n,k) \in \mathbb{Z}_{<0}[n,k]$.
Since $1 \le k \le 2\sqrt{n}/3$, we obtain
\[
	\alpha(n,k) \ge \alpha_{+}(n,1) + \alpha_{-}(n,2\sqrt{n}/3)
	=: \beta( \sqrt{n} ) \in \mathbb{Q}[\sqrt{n}].
\]
Using \textsc{Sturm} sequences (cf.~\cite[Theorem~4.1.10]{cohen}), it is easy to see that $\beta$ has no roots for $\sqrt{n} \ge 28$,
so $\alpha(n,k) \ge \beta( \sqrt{n} ) \ge 0$. Thus, the claim holds for $n \ge 28^2$.
For $9 \le n \le 28^2-1$ and $1 \le k \le \lfloor 2\sqrt{n}/3 \rfloor$,
we check the claim case by case.
\end{proof}

Using this result
we can now prove the first claim of Theorem~\ref{thm:InvToCyc}.

\begin{proof}[Proof of Theorem~\ref{thm:InvToCyc} (1)]
By Lemma~\ref{la:TripConcave} it suffices to check the inequality for $k = 1$ and $k = \lfloor 2\sqrt{n}/3\rfloor$.
The first case is easy to verify\detailsinline{ ($\trip(n,1) = \frac{4(n-2)}{n(n-1)} \ge \frac{4\cdot5}{6n}$)}, so consider the second case.
Note that
\[
	\frac{\trip(n,k)}{\trip(n+1,k)} = \frac{(n+1)(n-4k+2)}{(n-2k+1)^2} = 1 + \frac{n - (4k^2-1)}{(n-2k+1)^2},
\]
so $\trip(n,k)$ increases in $n$ for $n\le 4k^2-1$,
which holds for $n \geq 39$. %
\details{
Let $j := \lfloor \frac{2}{3} \sqrt{n} \rfloor$. Then we have
$4j^2 - 1 \ge 4 ( \frac{2}{3}\sqrt{n} - 1 )^2 - 1$ and
\[
	4 \left( \frac{2}{3}\sqrt{n} - 1 \right)^2 - 1 \ge n
	~~~ \Leftrightarrow ~~~
	n - \frac{48}{7}\sqrt{n} + \frac{27}{7} \ge 0,
\]
which holds for $n \ge 39$.
}%
We consider this case first.
Since $n \geq \lceil 9k^2/4 \rceil$, we see
\begin{align*}
	\trip(n, k) & \geq \trip\left( \left\lceil \frac{9}{4}k^2 \right\rceil, k \right) \\
	\detailsalign{= \frac{4k^2 ( \lceil \frac{9}{4}k^2 \rceil - 2k )!^2}
	{(\lceil \frac{9}{4}k^2 \rceil)!(\lceil \frac{9}{4}k^2 \rceil - 4k + 1)!}}
	&= \frac{4k^2}{(\lceil \frac{9}{4}k^2 \rceil - 4k + 1)}
	\prod_{i=1}^{2k} \frac{\lceil \frac{9}{4}k^2 \rceil-4k+i}
	{\lceil \frac{9}{4}k^2 \rceil-2k+i}\\
	\detailsalign{\ge \frac{4k^2}{( \frac{9}{4}k^2 - 4k + 2)}
	\prod_{i=1}^{2k} \left( 1 - \frac{2k}{\lceil \frac{9}{4}k^2 \rceil-2k+i} \right)}
	&\ge \frac{4k^2}{( \frac{9}{4}k^2)}
	\prod_{i=1}^{2k} \left(1 - \frac{2k}{\frac{9}{4}k^2-2k+1}\right)
	= \frac{16}{9} \left(1 - \frac{2k}{\frac{9}{4}k^2-2k+1}\right)^{2k}.
\end{align*}
The claim follows by Lemma~\ref{la:MonoPower}, since $k \geq 4$. %
\details{
Lemma~\ref{la:MonoPower} yields
$\trip(n,\lfloor \frac{2}{3}\sqrt{n} \rfloor) \ge \frac{16}{9}
\Bigl(1 - \frac{2\cdot4}{\frac{9}{4}4^2-2\cdot4+1}\Bigr)^{2\cdot4} \ge \frac{13}{100} > \frac{10}{117} \ge \frac{10}{3n}$
for $n\ge39$.
}%
For $10 \le n \le 38$ we check $\trip(n,\lfloor 2\sqrt{n}/3 \rfloor) \ge 10/(3n)$ case by case.
Finally, for $n \leq 9$ we compute the proportion explicitly.
\end{proof}

Theorem~\ref{thm:InvToCyc} (1) shows that we can construct a $3$-cycle
by looking at $\mcO(n)$ conjugates of an involution with small support.
Unfortunately, considering that many conjugates would result in a final algorithm
with complexity
$\wt{\mcO}(n^2)$.
Thus we do not use this result to construct the $3$-cycles directly, but instead use
it as a lower bound for the proportion of non-commuting conjugates.

\begin{corollary} \label{cor:EnoughNonComm}
Let $9 \le n \in \mathbb{N}$, $1 \le k \le 2\sqrt{n}/3$, $0 < \veps < 1$,
$G \in \{ \Alt_n, \Sym_n \}$ and $s \in G$ a $k$-involution.
Let $Z := \bigl\lceil \frac{3n}{5} \lceil 3 \log \veps^{-1} \rceil \bigr\rceil$.
Then, with probability at least $1-\veps$,
a set of $Z$ random conjugates of $s$
contains at least $\lceil 3 \log \veps^{-1} \rceil$ elements
not commuting with $s$.
\end{corollary}
\begin{proof}
Use the proportion established in Theorem~\ref{thm:InvToCyc} (1) and \textsc{Chernoff}'s bound 
(Lemma~\ref{la:chernoff}) with $\delta := 1/2$.
\details{
The proportion of $k$-involutions not commuting
with $s$ is at least $10/(3n)$; for $n \geq 10$ this
follows by Theorem~\ref{thm:InvToCyc} (1), and for $n = 9$ and $k \in \{1,2\}$
it is easily verified directly.
Let $\delta := 1/2$. Then
$(1-\delta) 10T/(3n)  \ge \lceil 3 \log \veps^{-1} \rceil$
and $e^{-5\delta^2T/(3n)} \le \veps$.
Lemma~\ref{la:chernoff} completes the proof.
}
\end{proof}

Next we prove the second part of Theorem~\ref{thm:InvToCyc} by establishing a bound for the conditional
probability that two $k$-involutions $s$ and $r$ satisfy $| \supp r \cap \supp s | = 1$,
given that they do not commute.
Note that in this case $(sr)^2$ is a $3$-cycle, so we immediately obtain the following corollary.

\begin{corollary} \label{cor:TCgivenNC}
Let $9 \le n \in \mathbb{N}$, $1 \le k \le 2\sqrt{n}/3$, $0 < \veps < 1$,
$G \in \{ \Alt_n, \Sym_n \}$ and $s \in G$ a $k$-involution.
Let $Z := \lceil 3 \log \veps^{-1} \rceil$.
Then, with probability at least $1-\veps$, a set
of $Z$ random conjugates of $s$ not commuting with $s$ contains an
element $r$ such that $(sr)^2$ is a $3$-cycle.
\end{corollary}
\details{The probability that for all $1 \le i \le T$
the product $(sr_i)^2$ is not a $3$-cycle is at most
\[
	\biggl(1 - \frac{1}{3} \biggr)^T
	\le \biggl(1 - \frac{1}{3} \biggr)^{3\log(\veps^{-1})}
	= e^{3\log(\veps^{-1})\log(1-1/3)}
	\stackrel{\log(1+x) \le x}{\le} e^{-\log(\veps^{-1})}
	= \veps.
\]
}

\begin{proof}[Proof of Theorem~\ref{thm:InvToCyc} (2)]
Let $s$ be a fixed $k$-involution and denote by
$\Sigma$ the proportion of $k$-involutions $r$ such that $(sr)^2$ is a $3$-cycle
among all $k$-involutions not commuting with $s$.
The proportion $\Sigma$ can be computed explicitly for $n \leq 9$,
so assume in the following that $n \geq 10$.
\details{%
$\Sigma = 7/17 > 1/3$ for $(n,k) = (9,2)$
and $\Sigma = 2(n-2)/\binom{n}{2}$ for $k = 1$.
}%
Let $T := \{ t \in s^{\Sym_n} : | \supp t \cap \supp s | = 1 \}$
and $C := \{ c \in s^{\Sym_n} : | \supp c \cap \supp s | = 0 \}$.
Then $(st)^2$ is a $3$-cycle for every $t \in T$ and
$[s,c] = 1_G$ for every $c \in C$.
\details{For $9 \le n$ and $\frac{2}{3}\sqrt{n} \ge k \in \mathbb{N}$ we have $n \ge 4k$, so $C\neq\emptyset$.}
We find $|T| = 2k(n-2k) \inv (n-2k-1,k-1)$ and $|C| = \inv(n-2k,k)$,
so the conditional probability $\Sigma$ is bounded below by
\begin{align*}
\frac{|T|}{\inv (n,k) - |C|}
\detailsalign{= \frac{2k(n-2k)(n-2k-1)!2^kk!}{2^{k-1}(k-1)!(n-4k+1)!
\left( \frac{n!}{(n-2k)!} - \frac{(n-2k)!}{(n-4k)!} \right)}}
&= \frac{4k^2(n-2k)!^2}{(n-4k+1)(n!(n-4k)!-(n-2k)!^2)}.
\end{align*}
This term is greater or equal to $1/3$ if and only if
\begin{equation}
\label{eq:ConBound}
    \left( 1 + \frac{12k^2}{n-4k+1} \right) \prod_{i=1}^{2k} \frac{n-4k+i}{n-2k+i} \ge 1.
\end{equation}
\details{
\begin{center}
$\begin{array}{cccl}
&\frac{4k^2(n-2k)!^2}{(n-4k+1)(n!(n-4k)!-(n-2k)!^2)} & \ge & \frac{1}{3}\\
\Leftrightarrow & 4k^2 \frac{(n-2k)!}{(n-4k)!}
& \ge & \frac{n-4k+1}{3} \left( \frac{n!}{(n-2k)!} - \frac{(n-2k)!}{(n-4k)!} \right) \\
\Leftrightarrow & \left( 4k^2 + \frac{n-4k+1}{3} \right) \frac{(n-2k)!}{(n-4k)!}
& \ge & \frac{n-4k+1}{3} \frac{n!}{(n-2k)!} \\
\Leftrightarrow & \left( 1 + \frac{12k^2}{n-4k+1} \right) \prod_{i=1}^{2k} \frac{n-4k+i}{n-2k+i} & \ge & 1.
\end{array}$
\end{center}
}
Define $g(n,k) := \left( 1 + 12k^2/(n-4k+1) \right) \left( 1 - 2k/(n-2k+1) \right)^{2k}$;
the claim follows if $g(n,k) \ge 1$.
For this purpose, consider the derivative
\begin{align*}
	\diff{n}g(n,k)
    \detailsalign{= 4k^2\left(1-\frac{2k}{n-2k+1}\right)^{2k}
	\left( \frac{1+\frac{12k^2}{n-4k+1}}{(n-2k+1)(n-4k+1)} - \frac{3}{(n-4k+1)^2} \right)}
	\detailsalign{= 4k^2\left(1-\frac{2k}{n-2k+1}\right)^{2k}
	\frac{ n-4k+1 + 12k^2 - 3(n-2k+1) }{(n-2k+1)(n-4k+1)^2}}
	&= \frac{8k^2(-n + 6k^2+k-1)}{(n-2k+1)(n-4k+1)^2}\left(1-\frac{2k}{n-2k+1}\right)^{2k}.
\end{align*}

Note that $k \leq 2\sqrt{n}/3$ by assumption and hence $n \geq 9k^2/4$.

Assume first $n \geq 6k^2 + k - 1$. Then $\diff{n}g(n,k) \le 0$, and
$\lim_{n\to\infty} g(n,k) = 1$ implies $g(n, k) \geq 1$.
Now assume $9k^2/4 \leq n < 6k^2 + k - 1$ and $k \geq 36$. Then $\diff{n}g(n,k) > 0$, hence
\begin{align*}
    g(n,k) & \ge g\!\left(9k^2/4,k\right) \\
	& = \left( 1 + \frac{12k^2}{9k^2/4-4k+1} \right)
	\left( 1 - \frac{2k}{9k^2/4-2k+1} \right)^{2k} =: h(k)f(k).
\end{align*}
Since $h(k) \geq 57/9$ %
\details{
\[
	\left( 1 + \frac{12k^2}{\frac{9}{4}k^2-4k+1} \right)
	= \left( 1 + \frac{12}{\frac{9}{4}-\frac{4}{k}+\frac{1}{k^2}} \right)
	\ge \left( 1 + \frac{12}{\frac{9}{4}} \right)
	= \frac{57}{9}.
\]
}
and $f(k)$ increases for $k \geq 2$ by Lemma~\ref{la:MonoPower},
we get $g(n,k) \geq 57/9\cdot f(36) > 1$.

Finally, for $9k^2/4 \leq n < 6k^2 + k - 1$ and $1 \leq k \leq 35$ we verify inequality~\eqref{eq:ConBound}
case by case.
\details{
$\Sigma = 7/17 > 1/3$ for $(n,k) = (9,2)$
and $\Sigma = 7/18 > 1/3$ for $(n,k) = (9,1)$
}
\end{proof}

\subsection{Pre-bolstering elements}

Let $G = \Sym_n$ or $G = \Alt_n$, and let $c \in G$ be a $3$-cycle.
In the algorithm, we use pre-bolstering elements to construct a long cycle matching~$c$.
Recall that an element $r$ is 
pre-bolstering with respect to
$c$ if
\[
    r = (w,u,a_1,\ldots,a_{\alpha})(v,b_1,\ldots,b_{\beta})(\ldots)
\]
or
\[
    r = (w,u,a_1,\ldots,a_{\alpha},v,b_1,\ldots,b_{\beta})(\ldots)
\]
with $\supp c = \{u,v,w\}$ and $\alpha,\beta \ge 2$.
If $k = \alpha + \beta + 3$, we call the element $k$-pre-bolstering.
Note that $k \ge 7$.
Denote by $L_{c, G}(k)$ the number of $k$-pre-bolstering elements of $G$
with respect to $c$.

\begin{lemma} \label{la:GoodOfLength}
Let $7 \le k \le n \in \mathbb{N}$ and $G \in \{ \Alt_n, \Sym_n \}$.
Then we have $L_{c,\Sym_n}(k) = 12 (n-3)! (k-6)$ and
$L_{c,\Alt_n}(k) = 6 (n-3)! (k-6)$.
Moreover,
\[
	\frac{1}{|G|} \sum_{k=7}^n L_{c,G}(k) \ge \frac{2}{5n}.
\]
\end{lemma}
\begin{proof}
A standard counting argument yields the formulae for $L_{c,G}(k)$. %
\details{
Note that $\beta = k - 3 - \alpha$, so the possible values for $\alpha$ are
$2 \le \alpha \le k - 5$, giving $k-6$ choices for $\alpha$.
Let $G = \Sym_n$. Each choice of $\alpha$ leaves two choices for the form of $r$.
If the form is fixed, we have $\frac{(n-3)!}{(n-k)!}$ possibilities to choose
and arrange the $a_i$ and the $b_i$. On the remaining $n-k$ points, any permutation
may be chosen. Finally, we may rearrange the $c_i$ in $6$ different ways.
Thus, in total we find
$L_{c,\Sym_n}(k) = 12 (n-3)! (k-6)$.
For $k \le n-2$, the results for $G = \Alt_n$ are analogous
(on the remaining $n-k$ points, choose any permutation of appropriate sign).
For $k = n-1$ and $k=n$, only one of the two forms exists in $\Alt_n$, depending
on the parity of $n$. However, in these cases, since
we have $|\Alt_1| = 1 = 2\cdot\frac{(n-k)!}{2}$,
we find $L_{c,\Alt_n}(k) = \frac{1}{2} L_{c,\Sym_n}(k)$.
}%
Thus, for $G \in \{\Alt_n,\Sym_n\}$, we obtain
\begin{align*}
	\frac{1}{|G|}\sum_{k=7}^n L_{c,G}(k)
    & = \frac{6}{n} \left( 1 - \frac{8n-28}{(n-1)(n-2)} \right) \\
    & \geq \frac{6}{n} \left( 1 - \frac{8\cdot7-28}{(7-1)(7-2)} \right) = \frac{2}{5n}.
\end{align*} %
\details{
We obtain
\begin{eqnarray*}
	\frac{1}{|G|}\sum_{k=7}^n L_{c,G}(k) &=& \frac{12(n-3)!}{n!} \sum_{k=7}^n (k-6)
	=\frac{12}{n(n-1)(n-2)}\sum_{k=1}^{n-6} k\\
	&=& \frac{6}{n} \frac{(n-5)(n-6)}{(n-1)(n-2)}
	=\frac{6}{n} \left( 1 - \frac{8n-28}{(n-1)(n-2)} \right).
\end{eqnarray*}
Since the derivative
\begin{eqnarray*}
	\diff{n} \left( 1 - \frac{8n-28}{(n-1)(n-2)} \right)
	&=& \frac{8}{(n-1)^2(n-2)^2} \left( \left(n-\frac{7}{2}\right)^2 - \frac{15}{4} \right)\\
	&\stackrel{n \ge 7 > 7/2}{\ge}& \frac{8}{(n-1)^2(n-2)^2} \left( \frac{49-15}{4} \right)
	\ge 0,
\end{eqnarray*}
is non-negative, this yields
\[
	\frac{1}{|G|}\sum_{k=7}^n L_{c,G}(k) \stackrel{n\ge7}{\ge}
	\frac{6}{n} \left( 1 - \frac{8\cdot7-28}{(7-1)(7-2)} \right) = \frac{2}{5n},
\]
concluding the proof.
}%
\end{proof}

Using \textsc{Chernoff}'s bound, we obtain a terminating condition for
Algorithm \algBE.

\begin{proposition} \label{prop:LCGoodElements}
Let $7 \le n \in \mathbb{N}$, $G \in \{ \Alt_n, \Sym_n \}$,
$c \in G$ a $3$-cycle, $0 < \veps < 1$
and $1/2 < \alpha \le 4/5$.
Let $S = \left\lceil 5n
\max\left( \frac{25}{18} \lceil \frac{1}{2}\log_{\alpha} \veps \rceil,
\left(5/4\right)^4 \log \veps^{-1} \right)\right\rceil$.
The probability that among $S$ random elements
at least $\lceil \frac{1}{2}\log_{\alpha} \veps \rceil$ are
$k$-pre-bolstering with respect to $c$ for some $7 \le k \le n$
is at least $1-\veps$.
\end{proposition}
\begin{proof}
\details{
Choose $\delta := \frac{16}{25}$. Then
$(1-\delta) \frac{2}{5n} S \ge \lceil \frac{1}{2}\log_{\alpha} \veps \rceil$
and $e^{-\frac{\delta^2S}{5n}} \le \veps$.
On the one hand,
\[
    e^{-\delta^2(\frac{5}{4})^4 \log(\veps^{-1})} = \veps^1 = \veps,
\]
on the other hand,
\[
    e^{-\delta^2\frac{25}{18}\cdot \frac{1}{2} \log_\alpha(\veps)}
    = e^{-\frac{64}{225}\cdot \frac{\log(\veps)}{\log(\alpha)}}
    = \veps^{-\frac{64}{225}\cdot \frac{1}{\log(\alpha)}}
    \leq \veps
\]
if $-\frac{64}{225}\cdot \frac{1}{\log(\alpha)} \leq 1$,
so if $\alpha \leq e^{-\frac{64}{225}} < 4/5$.
}%
Use Lemma~\ref{la:GoodOfLength} and \textsc{Chernoff}'s bound with $\delta := 16/25$.
\end{proof}

The next proposition establishes the second bound: a lower bound on the proportion
of $k$-pre-bolstering elements in $G$ with $\alpha n \le k \le n$ among the
$k$-pre-bolstering elements with $7 \le k \le n$.
This ensures that \algLC constructs long cycles with high probability.

\begin{proposition} \label{prop:LCBigOrbits}
Let $9 \le n \in \mathbb{N}$, $G \in \{ \Alt_n, \Sym_n \}$,
$c = (c_1,c_2,c_3) \in G$ a $3$-cycle, $0 < \veps < 1$ and $3/4 \leq \alpha \le 4/5$.
Let $R = \lceil \frac{1}{2} \log_{\alpha} \veps \rceil$ and $r_1,\ldots,r_R \in G$
random elements such that $r_i$ is $k_i$-pre-bolstering with respect to $c$.
The probability that there is at least one $k_j$ with $k_j \ge \max ( 9, \lceil \alpha n \rceil + 1 )$
is at least $1 - \veps$.
\end{proposition}
\begin{proof}
We want to show that the proportion of $k_i$-pre-bolstering elements with $k_i \ge \lceil \alpha n \rceil + 1$
among all pre-bolstering elements is at least $1 - \alpha^2$.
For $n = 9$ we verify the claim directly, so assume in the following $n \geq 10$.
Then $\lceil \alpha n \rceil + 1 \geq 9$,
and we find
\begin{align*}
	\frac{\sum_{k=\max( \lceil \alpha n \rceil + 1, 9 )}^n L_{c,G}(k) }
		{\sum_{k=7}^n L_{c,G}(k) }
	\detailsalign{=
	\frac{\sum_{k= \lceil \alpha n \rceil + 1 }^n (k-6) }
		{\sum_{k=7}^n (k-6) }}
	\detailsalign{=
	\frac{\sum_{k= \lceil \alpha n \rceil - 5 }^{n-6} k }
		{\sum_{k=1}^{n-6} k }}
	&=
	\frac{ (n-6)(n-5) - 2\sum_{k=1}^{\lceil \alpha n \rceil - 6} k }{ (n-6)(n-5) }\\
	&=
	1 - \frac{ (\lceil \alpha n \rceil - 6)(\lceil \alpha n \rceil - 5) }{(n-6)(n-5)}\\
	\detailsalign{\detailsalt{\stackrel{\textcolor{details}{\lceil\alpha n\rceil\ge7,~n\ge7}}{\ge}}{\geq} 1 - \frac{(\alpha n-5)(\alpha n-4)}{(n-6)(n-5)} }
    & > 1 - \frac{\alpha^2(n-6)(n-5)}{(n-6)(n-5)}.
\end{align*} %
\details{
Since $\alpha \le \frac{4}{5} < \frac{5}{6}$, we obtain
$\alpha n - 5 < \alpha ( n - 6 )$
and
$\alpha n - 4 \le \alpha ( n - 5 )$,
yielding
\[
	(\alpha n - 5)(\alpha n - 4) \le \alpha( n - 6 )(\alpha n - 4) < \alpha^2(n-6)(n-5)
\]
and thus $1 - \frac{(\alpha n-5)(\alpha n-4)}{(n-6)(n-5)} > 1 - \alpha^2$.
}%
The claim now follows by a standard argument.
\details{
The probability that $k_i < \max ( 7, \lceil \alpha n \rceil + 1 )$
for all $1 \le i \le R$ is at most
$ (\alpha^2)^R = \alpha^{2\lceil \frac{1}{2} \log_{\alpha} \veps \rceil}
\detailsalt{\stackrel{\textcolor{details}{\alpha < 1}}{\le}}{\leq} \alpha^{\log_{\alpha} \veps} = \veps $.}
\end{proof}

\subsection{Common fixed points of $k$-cycles}

The final result ensures that we construct an $n$- or an $(n-1)$-cycle in \algGens
and thus find the correct degree of the group
with high probability.

\begin{theorem} \label{thm:CommonFPs}
Let $0 < \veps < 1$, $0 < \alpha < 1$ and $n, k, t \in \N$ with $\alpha n \leq k < n$ and
\[
	t \ge \frac{1}{\log\left((1-\alpha)^{-1}\right)}
	\left( \log n + \log \veps^{-1} \right).
\]
The probability that $t$ random $k$-cycles in $\Sym_n$
have a common fixed point is at most $\veps$.
\end{theorem}
\begin{proof}
Denote by $\prob_{\fix}(n,k,t)$ the probability that $t$ random $k$-cycles in $\Sym_n$ have a common fixed point.
Let $r \in \Sym_n$ be a $k$-cycle and $1 \le m_1, \ldots, m_j \le n$ pairwise different points.
If $r$ fixes each of the $m_i$, then the probability that another random point $m_{j+1}$ is fixed by $r$
equals $(n-k-j)/(n-j) = 1 - k/(n-j)$.
Thus, the probability that $m_1,\ldots,m_{j+1}$ are common fixed points of $t$ random $k$-cycles equals
\[
	\prod_{i=0}^j \left(1-\frac{k}{n-i}\right)^t.
\]
Define $c_j := (-1)^j \binom{n}{j+1} \prod_{i=0}^j \Bigl( 1 - k/(n-i)\Bigr)^t$
(note that $c_j = 0$ for $j \geq n-k$);
a standard inclusion-exclusion principle shows $\prob_{\fix}(n,k,t) = \sum_{j=0}^{n-k-1} c_j$.
We will prove
\[
    \left|\frac{c_j}{c_{j+1}}\right| = \frac{j+2}{n-j-1}\cdot \left(1 - \frac{k}{n-j-1}\right)^{-t} \geq 1
\]
for $j+1 < n-k$.
To this end, note that
\begin{multline*}
	t \ge \frac{1}{ \log \left( (1-\alpha)^{-1} \right) } \log \left(\frac{n-2}{2}\right)
	\ge
	\frac{1}{ \log \left( (1-\alpha)^{-1} \right) } \log \left( \frac{n-j-1}{j+2} \right)\\
	=
	\frac{ \log \left( \frac{j+2}{n-j-1} \right) }{ \log
        \detailsalt{
            \textcolor{details}{\underbrace{\textcolor{black}{\left( \frac{ (n-j-1)(1-\alpha) }{n-j-1} \right) }}_{ \ge \frac{n-j-1-k}{n-j-1} }}
        }{
            \left( \frac{ (n-j-1)(1-\alpha) }{n-j-1} \right)
        }
    }
	\ge
	\frac{ \log \left( \frac{j+2}{n-j-1} \right) }
	{ \log \left( \frac{n-j-1-k}{n-j-1} \right) },
\end{multline*}
thus $ t \cdot \log \left( 1 - k/(n-j-1) \right) \le \log \left( (j+2)/(n-j-1) \right) $.
This implies
\[
	\left( 1 - \frac{k}{n-j-1} \right)^t \le \frac{j+2}{n-j-1}
\]
and hence $|c_j| \geq |c_{j+1}|$.

Since $c_j$ has alternating sign and $c_0$ is positive, this yields
$
	\sum_{j=0}^{n-k-1} c_j \le c_0
$.
Moreover,
\[
	t \ge \frac{1}{\log\left( (1-\alpha)^{-1} \right)} \left( \log n + \log \veps^{-1} \right)
	= \frac{\log\left(\frac{\veps}{n}\right)}{\log(1-\alpha)}
	\detailsalt{\stackrel{\textcolor{details}{\alpha \le \frac{k}{n}}}{\ge}}{\geq} \frac{ \log\left(\frac{\veps}{n}\right)}{\log(1-\frac{k}{n})},
\]
hence
$\log \left(\veps/n\right) \ge t \cdot \log \bigl( 1 - k/n \bigr)$.
We obtain $c_0 = n (1-k/n)^t  \leq \veps$, thus proving the claim.
\end{proof}

\section*{Acknowledgements}
We thank the anonymous referee for many helpful suggestions.

We acknowledge support from the following DFG grants:
SPP 1388 (first author),
Graduiertenkolleg Experimentelle und konstruktive Algebra at RWTH Aachen University (second author),
SPP 1489 (third and fourth author)
and the Australian grant: ARC DP110101153 (third author).

\bibliographystyle{plain}
\bibliography{bban}

\end{document}